\documentclass[10pt]{article}

\usepackage{amsmath,amssymb,amsthm,graphicx,mathrsfs,url}
\usepackage[usenames,dvipsnames]{color}

\definecolor{darkred}{rgb}{0.4,0.1,0.1}
\usepackage[colorlinks=true,linkcolor=darkred,citecolor=Blue]{hyperref}

\newtheorem{theorem}{Theorem}[section]         
\newtheorem{lemma}[theorem]{Lemma}
\newtheorem{corollary}[theorem]{Corollary}

\theoremstyle{definition}
\newtheorem{definition}[theorem]{Definition}

\newtheorem{remark}[theorem]{Remark}

\numberwithin{equation}{section}

\newcommand{\paragraf}{\textsection}
\renewcommand{\emptyset}{\varnothing}

\newcommand{\R}{\ensuremath{\mathbb R}}    
\newcommand{\C}{\ensuremath{\mathbb C}}    
\newcommand{\N}{\ensuremath{\mathbb N}}    

\newcommand{\gperp}{{[\perp]}}
\newcommand{\iso}{\circ}
\newcommand{\product}{[\cdot\,,\cdot]}
\newcommand{\hproduct}{(\cdot\,,\cdot)}

\newcommand{\calD}{\mathcal D}

\newcommand{\calG}{\mathcal G}
\newcommand{\calH}{\mathcal H}

\newcommand{\calK}{\mathcal K}
\newcommand{\calL}{\mathcal L}
\newcommand{\calM}{\mathcal M}
\newcommand{\calN}{\mathcal N}

\newcommand{\calP}{\mathcal P}

\newcommand{\calU}{\mathcal U}
\newcommand{\calV}{\mathcal V}
\newcommand{\calW}{\mathcal W}

\newcommand{\calZ}{\mathcal Z}

\newcommand{\veps}{\varepsilon}

\newcommand{\la}{\lambda}

\renewcommand{\Im}{\operatorname{Im}}

\newcommand{\linspan}{\operatorname{span}}
\renewcommand{\ker}{\operatorname{ker}}
\newcommand{\ran}{\operatorname{ran}}
\newcommand{\dom}{\operatorname{dom}}
\newcommand{\codim}{\operatorname{codim}}

\newcommand{\sap}{\sigma_{{ap}}}
\newcommand{\esap}{\wt\sigma_{{ap}}}

\newcommand{\spp}{\sigma_+}
\newcommand{\smm}{\sigma_-}

\newcommand{\Lra}{\Longrightarrow}
\newcommand{\Sra}{\Rightarrow}

\newcommand{\Llra}{\Longleftrightarrow}

\newcommand{\wto}{\rightharpoonup}
\newcommand{\Ato}{\stackrel{A}{\to}}
\newcommand{\Awto}{\stackrel{A}{\rightharpoonup}}

\newcommand{\ol}{\overline}
\newcommand{\ds}{\dotplus}
\newcommand{\wt}{\widetilde}

\renewcommand{\r}{\operatorname{r}}
\newcommand{\es}{\wt\sigma}
\newcommand{\er}{\wt\r}
\newcommand{\spip}{\sigma_{\pi_+}}
\newcommand{\spim}{\sigma_{\pi_-}}

\begin{document}
\title{Spectral Points of Type $\pi_+$ and Type $\pi_-$ of Closed Operators in Indefinite Inner Product Spaces}

\author{Friedrich Philipp and Carsten Trunk}

\maketitle

\begin{abstract}
We introduce the notion of spec\-tral poi\-nts of type $\pi_+$ and type $\pi_-$ of closed operators $A$ in a Hilbert space which is equipped with an indefinite inner product. It is shown that these points are stable under compact perturbations. In the second part of the paper we assume that $A$ is symmetric with respect to the indefinite inner product and prove that the growth of the resolvent of $A$ is of finite order in a neighborhood of a real spectral point of type $\pi_+$ or $\pi_-$ which is not in the interior of the spectrum of $A$. Finally, we prove that there exists a local spectral function on intervals of type $\pi_+$ or $\pi_-$.
\end{abstract}

\bigskip
\noindent\textit{Keywords}: Indefinite inner product; selfadjoint operator; spectrum of positive and negative type; spectrum of type $\pi_+$ and $\pi_-$; local spectral function; perturbation theory

\noindent\textit{MSC 2010}: 47A10, 47B50, 46C20, 47A55

\section{Introduction}
Let $(\calH,\product)$ be a Krein space and let $A$ be a bounded
or unbounded linear operator in $\calH$ which is selfadjoint with
respect to the Krein space inner product $[\cdot,\cdot]$. The
spectral properties of selfadjoint operators in Krein spaces differ
essentially from the spectral properties of selfadjoint operators in
Hilbert spaces, e.g., the spectrum $\sigma(A)$ of $A$ is in general
not real and even $\sigma(A)=\mathbb C$ may happen.

The indefiniteness of the scalar product $\product$ on $\calH$ induces a natural
classification of isolated real eigenvalues: A real isolated eigenvalue $\lambda_0$
 is said to be of {\em positive} ({\em negative})
 {\em type} if all corresponding eigenvectors are positive (negative, respectively)
 with respect to  $\product$. In this case,
there is no Jordan chain of length greater than one.
This classification of real isolated eigenvalues is  used frequently in some
papers from theoretical physics,
see, e.g., \cite{BBJ1,BBJ2,CGS05,GSZ,LT04}.

There is a corresponding notion for points from the approximate point spectrum $\sigma_{ap}(A)$. Namely,
a point $\lambda \in \sigma_{ap}(A)$ is called a
spectral point of {\em positive\/} {\rm (}{\em negative\/}{\rm )}
 {\em type of $A$\/} if for every approximate
eigensequence
$(x_{n})$  of $A$ at $\lambda$ we have
\begin{equation}\label{e:wieimmer}
\liminf_{n \to \infty}\, [x_{n},x_{n}] >0 \;\;\;\;
 \bigl(\,resp.\,\, \limsup_{n \to \infty}\, [x_{n},x_{n}] <0 \bigr).
\end{equation}
The above definitions also make sense when the underlying inner product
is no longer a Krein space inner product.
In \cite{lmm}
bounded operators in a Hilbert space $(\calH,\hproduct)$
are considered which are selfadjoint with respect to an  inner product $\product = (G\cdot,\cdot)$ with a selfadjoint bounded  operator $G$.
Note that in this case the point zero is allowed to be a point of the spectrum of $G$,
which corresponds to the situation where  $(\calH,\hproduct)$ is not a Krein space.
 In \cite{lmm} it is shown that the sets of the spectral points
 of positive and negative type are contained in $\R$.
 If, in addition, the non-real points of a neighbourhood of  spectral points of
 positive/negative type are contained in $\rho(A)$, then
  there exists a local spectral function $E$, see  \cite{lmm}.
 The second main result in \cite{lmm} is for a compact
 and $\product$-selfadjoint perturbation $K$: A spectral point of positive type
 which is not in the interior of $\sigma(A)$ and of $\sigma(A + K)$
  is either a spectral point of positive type or a regular point of $A + K$ or it is contained in $\sigma_{-,f}(A + K)$, see \cite{lmm}.

In \cite{ajt} the notions of spectral points of positive/negative type are generalized
 to spectral points of type $\pi_+$ and type $\pi_-$. These points are also introduced via approximate eigensequences, and the relation \eqref{e:wieimmer} is only required
for sequences $(x_n)$ in a subspace of finite codimension.
In \cite{ajt} the operator $A$ is allowed to be unbounded,
 but  $\product$ is still a Krein space inner product.
  One of the main results in \cite{ajt} is that the above-mentioned set $\sigma_{-,f}(A)$ essentially coincides with the set of the spectral points of type $\pi_+$ which are not of positive type.
   Moreover, a local spectral function similar as above is constructed.
   However, the proof relies on the  Krein space structure.
This paper is in a sense  continued by \cite{bpt}.
Moreover, in \cite{abjt}, the stability results from \cite{ajt} and \cite{lmm} were generalized to closed linear relations in Krein spaces and were used in, e.g., \cite{adp,BP,JT,JTW,PST}

In the present paper we drop the condition that $(\calH,\product)$ is a Krein space
 and (contrary to \cite{lmm}) allow the operator $A$ to be unbounded.
Some of the known results from \cite{ajt} and \cite{bpt}
 still hold in this much more general situation. They are
 collected in Section \ref{ss:first}. In addition, it is shown in
 Section \ref{ss:first} that $\ker (A-\lambda)$ is an Almost
 Pontryagin space for all complex numbers $\lambda$ from the spectrum
 of type $\pi_+$ or $\pi_-$. Moreover it is  shown that the spectral points
  of type $\pi$ are stable under compact perturbations.
   In Section \ref{binkrank1}
 it is proved in Theorem \ref{t:big_one}
  that a compact interval of type $\pi_+$ or $\pi_-$ is always contained in an open set $\calU$ such that $\calU$ either consists only of eigenvalues of $A$ or $\calU\setminus\R$ consists only of points outside $\sap(A)$.
  Here we also show that in this situation either each point
  of $\calU$ possesses a Jordan chain of infinite length
  or that there exists at most
  finitely many points in $\calU$ with a Jordan chain of length greater
  than one which has, in addition, a finite length.
  In Subsection \ref{ss:frp}, a finite rank perturbation is constructed which turns a spectral point of type $\pi$ into a spectral point of definite type.
  If  $\calU\setminus\R\subset \rho(A)$ then  the growth of the resolvent towards the interval is of finite order (see Theorem \ref{t:resolvente}).
Finally, we prove in Section \ref{binkrank2} that the operator $A$ possesses a local spectral function on intervals of  type $\pi_+$ or type $\pi_-$.

\section{Preliminaries}
In this paper let $(\calH,\hproduct)$ be a Hilbert space and let $G$ be a bounded selfadjoint operator in $\calH$. By $\product$ we denote the inner product which is induced by $G$, i.e.
\begin{equation}\label{Geraberg}
[x,y] := (Gx,y) \;\text{ for }\; x,y\in\calH.
\end{equation}
The operator $G$ is called the {\it Gram operator} of the inner product $\product$ in the Hilbert space $(\calH,\hproduct)$.

A vector $x\in\calH$ is called {\it positive} ({\it negative}, {\it neutral}\,) if $[x,x] > 0$ ($[x,x] < 0$, $[x,x]=0$, respectively). A subset is called {\it positive} ({\it negative},
{\it neutral}\,) if all its non-zero vectors are positive (negative, neutral, respectively). As usual (see e.g.\ \cite{ai,b}), the {\it orthogonal companion} $\calM^\gperp$ and
the {\it isotropic part} $\calM^\iso$ of a subset $\calM$ are defined by
$$
\calM^\gperp := \{x\in\calH : [x,y] = 0
\;\text{ for all }\; y\in\calM\}\quad\text{and}\quad\calM^\iso := \calM\cap\calM^\gperp.
$$
In this paper a subspace is always a closed linear manifold.
Let $\calL\subset\calH$ be a subspace. A {\it fundamental decomposition of $\calL$} is a decomposition of the type
\begin{equation}\label{e:fd}
\calL = \calL_+ [\ds] \calL_- [\ds] \calL^\iso,
\end{equation}
where $\calL_+$ is a positive subspace, $\calL_-$ is a negative subspace,
 and the projections in $\calL$ onto $\calL_+$, $\calL_-$ and $\calL^\iso$,
 which are defined by this decomposition, are bounded operators. The symbol
  $[\ds]$ indicates that the sum is direct and orthogonal with respect to the
   inner product $\product$. Recall, that a subspace $\calL\subset\calH$
    always admits a fundamental decomposition \eqref{e:fd}
    (see \cite[Theorem IV.5.2]{b}). The numbers
$$
\kappa_+(\calL) := \dim\calL_+,\quad\kappa_-(\calL) := \dim\calL_-,
\quad\text{and }\kappa_0(\calL) := \dim\calL^\iso
$$
will be called the {\it rank of positivity}, {\it rank of negativity} and the {\it rank of degeneracy} of $\calL$, respectively. They do not depend on the particularly chosen fundamental decomposition. Furthermore, we define
$$
\kappa_{+,0}(\calL) := \kappa_+(\calL) + \kappa_0(\calL)
\quad\text{and}\quad
\kappa_{-,0}(\calL) := \kappa_-(\calL) + \kappa_0(\calL),
$$
and call these values the {\it rank of non-negativity} and the {\it rank of non-positivity} of $\calL$, respectively.
A subset $\calL$ is called {\it uniformly positive}
({\it uniformly negative}) if there exists a number $\delta > 0$ such that
$$
[x,x] \ge \delta\|x\|^2 \;\text{ for all }\; x\in\calL
\qquad
(-[x,x]\ge\delta\|x\|^2 \;\text{ for all }\; x\in\calL,\text{ respectively}).
$$
A subset is called {\it uniformly definite} if it is either uniformly positive or uniformly negative. Recall, that for a uniformly definite subspace $\calL\subset\calH$ we have
(see, e.g., \cite{lmm})
\begin{equation}\label{Elgersburg}
\calH = \calL[\ds]\calL^\gperp.
\end{equation}

Let $A$ be a closed and densely defined operator in $\calH$.
We denote the spectrum and the resolvent set by $\sigma(A)$ and $\rho(A)$, respectively.
By $\ker(A)$ we denote the kernel and by $\ran(A)$ the range of $A$.
 We call $A$ a $\Phi_+$-operator if $\ker(A)$ is finite-dimensional and $\ran(A)$ is closed.
 Recall (see, e.g., \cite[Theorem 8 in Section 16]{m}) that  $A$ is a
 $\Phi_+$-operator if and only if there exist a subspace $\calM\subset\calH$ with
 $\codim\,\calM < \infty$ and a number $c > 0$ such that
\begin{equation}\label{delta.delta}
\|Ax\| \ge c\|x\| \;\text{ for all }\; x\in\calM\cap\dom A.
\end{equation}

 The {\it approximate point spectrum $\sap(A)$ of $A$} is the set
  of all points $\lambda\in\C$ for which there exists a sequence $(x_n)$ in $\dom A$ with the property
$$
\|x_n\| = 1, \;n\in\N, \quad\text{and}\quad (A - \lambda)x_n\to 0 \;\text{ as }n\to\infty.
$$
A point $\lambda\in\sap(A)$ is called an {\it approximate eigenvalue of $A$}. If $\lambda\in\C$ is not an approximate eigenvalue of $A$, it is called a {\it point of regular type of $A$}. We denote the set of those points by $\r(A)$. It is not difficult to see that $\lambda\in\C$ is a point of regular type of $A$ if and only if $\ker(A - \lambda) = \{0\}$ and $\ran(A - \lambda)$ is closed. In particular, if $\lambda\in\r(A)$, then $A - \lambda$ is a $\Phi_+$-operator.

As usual, the compactification
$\mathbb C \cup \{\infty\}$
of $\mathbb C$ is denoted  by $\overline{\mathbb C}$.
We define $\esap(A):=\sap(A)$ if $A$ is bounded, and $\esap(A):= \sap(A)\cup\{\infty\}$ if $A$ is unbounded and call the set $\esap(A)$ the {\it extended approximate point spectrum of $A$}. The {\it extended spectrum $\es(A)$ of $A$} is defined analogously. The complementary sets
$$
\wt\rho(A) := \ol\C\setminus\es(A)
\quad\text{and}\quad
\er(A) := \ol\C\setminus\esap(A)
$$
are called the {\it extended resolvent set} and the {\it extended set of points of regular type of $A$}, respectively. Obviously, $\sap(A)\subset\sigma(A)$ and $\esap(A)\subset\es(A)$. Moreover, we have
\begin{equation}\label{e:partial}
\partial\sigma(A)\subset\sap(A)
\quad\text{and}\quad
\partial\es(A)\subset\esap(A).
\end{equation}
A point $\lambda\in\ol\C$ is contained in $\er(A)$ if and only if there exist an open neighborhood $\calU$ of $\lambda$ in $\C$ and a number $c>0$ such that
\begin{equation}\label{e:rA}
\|(A - \mu)x\| \ge c\|x\| \;\text{ for all }\;
\mu\in\calU\setminus\{\infty\} \;\text{ and all }\;
x\in\dom A.
\end{equation}
Thus, $\er(A)$ and $\r(A)$ are open in $\ol\C$ and $\C$, respectively.

For a linear manifold $\calL\subset\calH$ the {\it codimension of $\calL$} is defined by $\codim\calL := \dim(\calH/\calL)$. If $\calM\subset\calH$ is another linear manifold
 such that $\calL\subset\calM$ we define $\codim_\calM\calL := \dim(\calM/\calL)$.

\section{Spectral Points of Type \texorpdfstring{$\pi_+$}{pi+} and Type \texorpdfstring{$\pi_-$}{pi-}}\label{ss:first}
Throughout this section, let $A$ be a closed, densely defined operator in $\calH$.
We define the spectral points of type $\pi_+$ and type $\pi_-$ of $A$ in analogy to
\cite{abjt,ajt,bpt}.
 However, we emphasize that here neither $(\calH,\product)$ is
 assumed to be a Krein space
 (as in \cite{abjt,ajt,bpt})
 nor is the operator $A$ assumed to be selfadjoint (as in \cite{ajt,bpt}).
The following definition  is a generalization of the spectral
points of definite type (see, e.g., \cite{lmm,p}).

\begin{definition}\label{d:spip}
Let $A$ be a closed, densely defined operator in $\calH$.
A point $\lambda\in\sap(A)$ is called a {\it spectral point of type $\pi_+$} ({\it type $\pi_-$}) {\it of $A$} if there exists a linear manifold $\calH_\lambda\subset\calH$ with $\codim\,\calH_\lambda < \infty$, such that for every sequence $(x_n)$ in $\calH_\lambda\cap\dom A$ with $\|x_n\|=1$ and $(A - \lambda)x_n\to 0$ as $n\to\infty$ we have
$$
\liminf_{n\to\infty}\,[x_n,x_n] > 0 \quad
\left(\limsup_{n\to\infty}\,[x_n,x_n] < 0,\text{ respectively}\right).
$$
The point $\infty$ is called a {\it spectral point of type $\pi_+$} ({\it type $\pi_-$}) {\it of $A$} if $A$ is unbounded and if there exists a linear manifold $\calH_\lambda\subset\calH$ with $\codim\calH_\lambda < \infty$, such that for every sequence $(x_n)$ in $\calH_\lambda\cap\dom A$ with $\|Ax_n\| = 1$ and $x_n\to 0$ as $n\to\infty$ we have
$$
\liminf_{n\to\infty}\,[Ax_n,Ax_n] > 0 \quad
\left(\limsup_{n\to\infty}\,[Ax_n,Ax_n] < 0,\text{ respectively}\right).
$$
We denote the set of all spectral points of type $\pi_+$ (type $\pi_-$) of $A$ by $\spip(A)$ ($\spim(A)$, respectively).

A point $\lambda\in\spip(A)$ ($\lambda\in\spim(A)$) is called a {\it spectral point of positive type} ({\it negative type}, respectively) {\it of $A$} if $\calH_\lambda$ in the above definition can be chosen as $\calH$. The set consisting of all spectral points of positive (negative) type of $A$ is denoted by $\spp(A)$ ($\smm(A)$, respectively).
\end{definition}

\begin{remark} Contrary to the notion above, in \cite{abjt,ajt,bpt}
the notion $\sigma_{++}(A)$ and $\sigma_{--}(A)$
is used for spectral points of positive (negative) type of $A$. However, here we
will use the notion $\spp(A)$ ($\smm(A)$, respectively) as in \cite{lmm}.
\end{remark}

\begin{remark}\label{r:PhiPlus}
If $\lambda\in\C$ then $A - \lambda$ is a $\Phi_+$-operator if and only if $\lambda\in(\spip(A)\cap\spim(A))\cup\r(A)$. Indeed, if $A - \lambda$ is a $\Phi_+$-operator, then there is a subspace $\calH_\lambda$ with finite codimension such that there does not
exist any sequence $(x_n)$ in $\calH_\lambda\cap\dom A$ with $\|x_n\|=1$ and $(A - \lambda)x_n\to 0$ as $n\to\infty$. The opposite direction follows directly from Definition \ref{d:spip}.
\end{remark}

In the sequel, by $\calH_A$ we denote the Hilbert space $(\dom A,\hproduct_A)$, where
$$
(x,y)_A := (x,y) + (Ax,Ay), \quad x,y\in\dom A.
$$
The graph norm on $\calH_A$ induced by $\hproduct_A$ is denoted by $\|\cdot\|_A$, i.e.\
\begin{equation}\label{PoeHoe}
\|x\|_A := \sqrt{\|x\|^2 + \|Ax\|^2}, \quad x\in\dom A.
\end{equation}
For $M\subset\calH_A$ we denote the closure of $M$ in $\calH_A$ by $\ol{M}^A$. If $(x_n)$ is a sequence in $\calH_A$ converging (weakly) to some $x\in\calH_A$, we write $x_n\Ato x$ ($x_n\Awto x$, respectively), $n\to\infty$.
In the following theorem we collect different characterizations for
a point to belong to $\spip(A)$ {\rm (}or to $\lambda\in\spim(A)${\rm )},
see also \cite{abjt,ajt}.

\begin{theorem}\label{t:closed subspace}
Let $A$ be a closed, densely defined operator in $\calH$ and let
 $\lambda\in\esap(A)$. Then the following statements are equivalent.
\begin{enumerate}
\item[{\rm (i)}] $\lambda\in\spip(A)$ {\rm (}$\lambda\in\spim(A)${\rm )}.
\item[{\rm (ii)}] There exists a linear manifold $\calD_\lambda\subset\dom A$ with
$\codim_{\dom A}\calD_\lambda < \infty$, such that for every sequence
$(x_n)$ in $\calD_\lambda$ we have: If $\lambda\neq\infty$, then
\begin{equation}\label{e:ae}
\|x_n\|=1 \quad\text{and}\quad (A - \lambda)x_n\to 0 \;\text{ as }n\to\infty
\end{equation}
implies
\begin{equation}\label{e:liminf}
\liminf_{n\to\infty}\,[x_n,x_n] > 0 \quad
\left(\limsup_{n\to\infty}\,[x_n,x_n] < 0,\text{ respectively}\right).
\end{equation}
If $\lambda=\infty$, then
\begin{equation}\label{e:aei}
\|Ax_n\| = 1 \quad\text{and}\quad x_n\to 0 \;\text{ as }n\to\infty
\end{equation}
implies
\begin{equation}\label{e:liminfi}
\liminf_{n\to\infty}\,[Ax_n,Ax_n] > 0 \quad
\left(\limsup_{n\to\infty}\,[Ax_n,Ax_n] < 0\text{ respectively}\right).
\end{equation}
\item[{\rm (iii)}] There exists a linear manifold
 $\wt\calD_\lambda\subset\dom A$ with $\codim_{\dom A}\wt\calD_\lambda < \infty$ which is closed in $\calH_A$ such that for every sequence $(x_n)$ in $\wt\calD_\lambda$ we have: If $\lambda\neq\infty$, then {\rm \eqref{e:ae}} implies {\rm\eqref{e:liminf}}. If $\lambda=\infty$, {\rm\eqref{e:aei}} implies {\rm\eqref{e:liminfi}}.
\item[{\rm (iv)}] There exists a subspace\footnote{Recall
 that here a subspace is always a closed linear (sub)manifold}
 $\calH_\lambda\subset\calH$ with $\codim\calH_\lambda < \infty$ such that for every sequence $(x_n)$ in $\calH_\lambda\cap\dom A$ we have: If $\lambda\neq\infty$, then {\rm\eqref{e:ae}} implies {\rm\eqref{e:liminf}}. If $\lambda=\infty$, {\rm\eqref{e:aei}} implies {\rm\eqref{e:liminfi}}.
\item[{\rm (v)}] If $\lambda\neq\infty$, then for every sequence $(x_n)$ in $\dom A$ with $x_n\wto 0$ as $n\to\infty$ {\rm\eqref{e:ae}} implies {\rm\eqref{e:liminf}}. If $\lambda=\infty$, then for every sequence $(x_n)$ in $\dom A$ with $Ax_n\wto 0$ as $n\to\infty$ {\rm\eqref{e:aei}} implies {\rm\eqref{e:liminfi}}.
\end{enumerate}
\end{theorem}
\begin{proof}
Let $\lambda\in\spip(A)$. A similar reasoning applies to $\lambda\in\spim(A)$.

(i)$\Sra$(ii). Let $\calH_\lambda$ be a linear manifold
 with finite codimension as in Definition \ref{d:spip}.
 Then there exists a finite-dimensional subspace $\calZ\subset\calH$, such that
$$
\calH = \calH_\lambda \ds \calZ \quad\text{and}\quad
\dom A = (\dom A\cap\calH_\lambda) \ds (\dom A\cap\calZ),
$$
see, e.g., \cite[\S 7.6]{ko}.
Thus, $\calD_\lambda := \dom A\cap\calH_\lambda$ is a linear manifold as in statement (ii).

(ii)$\Sra$(iii). Let $\calD_\lambda$ be a linear manifold as in (ii). In order to show (iii), we set $\wt\calD_\lambda := \ol{\calD_\lambda}^A$,
 where $\ol{\calD_\lambda}^A$ denotes the closure of $\calD_\lambda$ with respect
 to the graph norm in \eqref{PoeHoe}.
 Let $(x_n)$ in $\wt\calD_\lambda$ be a sequence satisfying \eqref{e:ae} if $\lambda\neq\infty$ or \eqref{e:aei} if $\lambda = \infty$. Then there is
  a sequence $(u_n)$ in $\calD_\lambda$ with $\|x_n - u_n\|\to 0$ and $\|Ax_n - Au_n\|\to 0$ as $n\to\infty$. If $\lambda\neq\infty$, we have $\|u_n\|\to 1$ and $(A - \lambda)u_n\to 0$ as $n\to\infty$, which implies
$$
\liminf_{n\to\infty}\,[x_n,x_n] = \liminf_{n\to\infty}\,[u_n,u_n] > 0.
$$
If $\lambda = \infty$, $u_n\to 0$ and $\|Au_n\| \to 1$ as $n\to\infty$ follows, which yields
$$
\liminf_{n\to\infty}\,[Ax_n,Ax_n] = \liminf_{n\to\infty}\,[Au_n,Au_n] > 0.
$$
This shows (iii).

(iii)$\Sra$(iv) \& (v)$\Sra$(iv). Suppose that (iv) is not true.
 If $\lambda\neq\infty$, then for any subspace $\calM$ of $\calH$
 with finite codimension there is a sequence $(x_{n,\calM})\subset\calM
 \cap\dom A$ with $\|x_{n,\calM}\|=1$, $n\in\N$, $(A-\lambda)x_{n,\calM} \to 0$,
 $n\to\infty$, and $\liminf_{n\to\infty}[x_{n,\calM},x_{n,\calM}] \leq 0$.
 Hence, by induction, we find a sequence $(x_n)$ in $\dom A$
 with $\|x_n\|=1$, $x_n\in \{x_1, \ldots, x_{n-1}\}^\perp$,
 $\|(A-\lambda)x_n\|\leq \frac{1}{n}$ and $[x_n,x_n]\leq \frac{1}{n}$. 
Therefore the orthonormal sequence $(x_n)$ satisfies
$$
(A - \lambda)x_n\to 0 \;\text{ as }n\to\infty\quad\text{and}\quad\liminf_{n\to\infty}\,[x_n,x_n] \le 0.
$$
 In the case $\lambda = \infty$ there exists a sequence $(x_n)$ in $\dom A$ with
$$
\|Ax_n\| = 1, \; x_n\to 0 \;\text{ as }n\to\infty \quad\text{and}\quad
\liminf_{n\to\infty}\,[Ax_n,Ax_n] \le 0.
$$
We may assume that $(Ax_n)$ converges weakly to some $y$. Since $A$ is a
 closed operator, it has a closed graph which  is also weakly closed.
  Then $\{x_n,Ax_n\}\wto\{0,y\}$, $n\to\infty$, in $\calH\times\calH$ and $y = 0$
 follows. Thus, in both cases we have $x_n\wto 0$ and $Ax_n\wto 0$, which yields
 the weak convergence of $(x_n)$ to zero in
  $(\calH_A,\hproduct_A)$.
  This shows that (v) does not hold. Suppose, that (iii) holds
   Then $\calH_A$ admits a decomposition
$$
\calH_A = \wt\calD_\lambda \ds \wt\calD
$$
with some finite-dimensional subspace $\wt\calD$. The projections onto $\wt\calD_\lambda$ and $\wt\calD$ with respect to this decomposition are continuous in $\calH_A$. Let $(u_n)$ in $\wt\calD_\lambda$ and $(v_n)$ in $\wt\calD$ be sequences such that $x_n = u_n + v_n$ holds. Since $\wt\calD$ is finite-dimensional, $v_n\Awto 0$ implies $v_n\Ato 0$, which means $v_n\to 0$ and $Av_n\to 0$ as $n\to\infty$. If $\lambda\neq\infty$, we have $\|u_n\|\to 1$, $(A - \lambda)u_n\to 0$ as $n\to\infty$ and
$$
\liminf_{n\to\infty}\,[u_n,u_n] = \liminf_{n\to\infty}\,\bigl(
[x_n,x_n] - [x_n,v_n] - [v_n,u_n]\bigr) \le 0,
$$
which is a contradiction to (iii). If $\lambda = \infty$, then $u_n\to 0$, $\|Au_n\|\to 1$ as $n\to\infty$ and
$$
\liminf_{n\to\infty}\,[Au_n,Au_n] = \liminf_{n\to\infty}\,\bigl(
[Ax_n,Ax_n] - [Ax_n,Av_n] - [Av_n,Au_n]\bigr) \le 0
$$
follows, contradicting (iii).

Obviously, (iv) implies (i) and, hence, assertions (i)--(iv) are equivalent.
It remains to show that (iv) implies (v). For this let
$\calH_\lambda$ be a subspace as in (iv). Then there exists a finite-dimensional
subspace $\calG_\lambda\subset\dom A$, such that
$$
\calH = \calH_\lambda \ds \calG_\lambda.
$$
Let $\lambda\neq\infty$ and let $(x_n)$ in $\dom A$ be a sequence with $x_n\wto 0$ as $n\to\infty$ which fulfils \eqref{e:ae}. Further, let $(u_n)$ in $\calH_\lambda$ and $(v_n)$ in $\calG_\lambda$ be sequences, such that $x_n = u_n + v_n$. Then, since $v_n\wto 0$, we have $v_n\to 0$, and thus $\|u_n\|\to 1$, $(A - \lambda)u_n\to 0$,
and
$$
\liminf_{n\to\infty}\,[x_n,x_n] = \liminf_{n\to\infty}\,[u_n,u_n] > 0.
$$
Suppose $\lambda = \infty$. If $(x_n)$ in $\dom A$ is a sequence with $Ax_n\wto 0$ as $n\to\infty$ such that \eqref{e:aei} holds, define the sequences $(u_n)$ in $\calH_\lambda$ and $(v_n)$ in $\calG_\lambda$ as above. Then $u_n\to 0$, $v_n\to 0$ and $Av_n\to 0$ as $n\to\infty$, which implies $\lim_{n\to\infty}\|Au_n\| = 1$ and therefore
$$
\liminf_{n\to\infty}\,[Ax_n,Ax_n] = \liminf_{n\to\infty}\,[Au_n,Au_n] > 0.
$$
Therorem \ref{t:closed subspace} is proved.
\end{proof}

The following lemma shows that the point $\infty$ cannot be a spectral point of type
$\pi_+$ or $\pi_-$ when it is not of positive (resp.\ negative) type.

\begin{lemma}\label{l:infty_piplus}
$\infty\in\spip(A)$ implies $\infty\in\spp(A)$,
$\infty\in\spim(A)$ implies $\infty\in\smm(A)$.
\end{lemma}
\begin{proof}
If $\infty\in\spip(A)\setminus\spp(A)$, then there exists $(x_n)$ in $\dom A$ with
$x_n\to 0$ as $n\to\infty$, $\|Ax_n\| = 1$ and $\liminf_{n\to\infty}\,[Ax_n,Ax_n] \le 0$. We may assume that $Ax_n\wto y$ as $n\to\infty$ for some $y\in\calH$. But then $\{x_n,Ax_n\}\wto\{0,y\}$ in $\calH\times\calH$, which implies $y = 0$, since $A$ is closed. By Theorem \ref{t:closed subspace} we obtain a contradiction to $\infty\in\spip(A)$.
\end{proof}

In the following we study compact sets which consist of points that either belong to $\spip(A)\cup\spim(A)$ or to $\er(A)$. As a byproduct, it will turn out that the sets $\spip(A)$, $\spim(A)$, $\spp(A)$ and $\smm(A)$ are relatively open in $\esap(A)$.
Theorems \ref{t:comp_set_pp} and \ref{t:comp_set} below can be proved
in the same way as in  \cite[Lemma 2 and Lemma 12]{ajt}.
Therefore, we omit their proofs.


\begin{theorem}\label{t:comp_set_pp}
Let $A$ be a closed, densely defined operator in $\calH$, and let $K\subset\ol\C$ be a compact set such that $K\cap\esap(A)\subset\spp(A)$ {\rm (}$K\cap\esap(A)\subset\smm(A)${\rm )}. Then there exist an open neighborhood $\calU$ in $\ol\C$ of $K$ and $\veps > 0$ such that
$$
\lambda\in\calU\setminus\{\infty\},\; x\in\dom A,\;
\|(A - \lambda)x\|\le\veps\|x\|
$$
implies
$$
[x,x]\ge\veps\|x\|^2 \quad (-[x,x]\ge\veps\|x\|^2,\text{ respectively}\,).
$$
In this case, we have
$$
\calU\cap\esap(A)\subset\spp(A)\quad (\calU\cap\esap(A)\subset\smm(A),\text{ respectively}).
$$
\end{theorem}

\begin{theorem}\label{t:comp_set}
Let $A$ be a closed, densely defined operator in $\calH$, and let $K\subset\ol\C$ be a compact set such that $K\cap\esap(A)\subset\spip(A)$ {\rm (}$K\cap\esap(A)\subset\spim(A)${\rm )}. Then there exist an open neighborhood $\calU$ in $\ol\C$ of $K$, a linear manifold $\calH_0\subset\calH$ with $\codim\calH_0 < \infty$, and $\veps > 0$ such that
$$
\lambda\in\calU\setminus\{\infty\},\; x\in\calH_0\cap\dom A,\;
\|(A - \lambda)x\|\le\veps\|x\|
$$
implies
$$
[x,x]\ge\veps\|x\|^2 \quad (-[x,x]\ge\veps\|x\|^2,\text{ respectively}\,).
$$
In this case, we have
$$
\calU\cap\esap(A)\subset\spip(A)\quad (\calU\cap\esap(A)\subset\spim(A),\text{ respectively}).
$$
\end{theorem}

Theorems \ref{t:comp_set_pp} and \ref{t:comp_set} in particular imply the following corollary.

\begin{corollary}
The sets $\spip(A)$, $\spim(A)$, $\spp(A)$ and $\smm(A)$ are {\rm (}relatively{\rm )} open in $\esap(A)$.
\end{corollary}

There is a certain connection between the linear manifold $\calH_\lambda$ from Definition \ref{d:spip} and the "nonpositive part" of the eigenspace $\ker(A - \lambda)$. For this, we first recall the notion of an Almost Pontryagin space,
see e.g.\ \cite{kww} and \cite{pt}.

\begin{definition}
A subspace $\calL\subset\calH$ is called an
{\it Almost Pontryagin space with finite rank of non-positivity
{\rm(}non-negativity{\rm)}} if there exists a uniformly positive
 (uniformly negative, respectively) subspace
 $\wt\calL\subset\calL$ with $\codim_{\calL}\wt\calL < \infty$.
\end{definition}

\begin{lemma}\label{l:ap}
Let $A$ be a closed, densely defined operator in $\calH$ and
let $\lambda\in\spip(A)\setminus\{\infty\}$ {\rm (}$\lambda\in\spim(A)\setminus\{\infty\}${\rm )}. Then $\ker(A - \lambda)$ is an Almost Pontryagin space with finite rank of non-positivity {\rm (}non-negativity, respectively{\rm )}.
\end{lemma}
\begin{proof}
We show Lemma \ref{l:ap} only for $\lambda\in\spip(A)$. Let $\calL_+[\ds]\calL_-[\ds]\calL^\iso$ be a fundamental decomposition of
 $\ker(A - \lambda)$, cf.\ \eqref{e:fd}.
If $\calL_-[\ds]\calL^\iso$ is infinite-dimensional then there is an
orthonormal  sequence in $\calL_-[\ds]\calL^\iso \subset \ker(A - \lambda)$,
which is
by Theorem \ref{t:closed subspace} (v) impossible.
Hence, $\calL_+$ has finite codimension in $\calL$, and
it remains to show that $\calL_+$ is uniformly positive.
Suppose, that this is not the case. Then there exists a sequence $(x_n)$
in $\calL_+$ with $\|x_n\| = 1$ and $\lim_{n\to\infty}[x_n,x_n] = 0$.
We may assume that $(x_n)$ converges weakly to some $x_0\in\calL_+$.
For this $x_0$ we have
$$
|[x_0,x_0]| \le |[x_0,x_0 - x_n]| + [x_0,x_0]^{1/2}[x_n,x_n]^{1/2},
$$
and $x_0 = 0$ follows. But this contradicts $\lambda\in\spip(A)$ by Theorem \ref{t:closed subspace} (v).
\end{proof}

The following theorem characterizes which linear manifolds can be
 chosen for $\calH_\lambda$ in Definition \ref{d:spip} and what
 their smallest possible codimension is. Theorem \ref{t:H0} is
 contained in \cite[Lemma 3.1, Theorem 3.3]{bpt} for the situation
 where $A$ is a selfadjoint operator in a Krein space. However,
 the proof in \cite{bpt} is also valid for the current situation,
 where $A$ is only a closed, densely defined operator in a space
 with inner product given by \eqref{Geraberg}. Therefore, we omit
 the proof of Theorem \ref{t:H0}.

\begin{theorem}\label{t:H0}
Let $A$ be a closed, densely defined operator in $\calH$ and
let $\lambda\in\spip(A)\setminus\{\infty\}$ {\rm (}$\lambda\in\spim(A)\setminus\{\infty\}${\rm )}. A linear manifold $\calH_\lambda$ with $\codim\calH_\lambda < \infty$ is as in Definition {\rm \ref{d:spip}} if and only if the subspace
\begin{equation*}\label{Himmel}
\ol{\calH_\lambda\cap\dom A}^A\cap\ker(A - \lambda)
\end{equation*}
is positive {\rm (}negative, respectively{\rm )}. In the case where $\calH_\lambda$ is closed, this is equivalent to the positivity {\rm (}negativity, respectively{\rm )} of
\begin{equation*}\label{Erde}
\calH_\lambda\cap\ker(A - \lambda).
\end{equation*}
Moreover, there exists a subspace $\calH_\lambda'$ with this property and
$$
\codim\calH_\lambda' = \kappa_{-,0}(\ker(A - \lambda))
\quad\left(
\codim\calH_\lambda' = \kappa_{+,0}(\ker(A - \lambda)),\text{ respectively}\right)
$$
which is the smallest possible codimension of all the linear manifolds
 satisfying the conditions of Definition {\rm \ref{d:spip}}.
\end{theorem}

\begin{corollary}\label{c:spp_in_spip}
Let $\lambda\in\spip(A)\setminus\{\infty\}$ {\rm (}$\lambda\in\spim(A)\setminus\{\infty\}${\rm )}. Then $\lambda\in\spp(A)$ {\rm (}$\lambda\in\smm(A)$, respectively{\rm )} if and only if $\ker(A - \lambda)$ is positive {\rm (}negative, respectively{\rm )}.
\end{corollary}

\begin{corollary}\label{Steinach}
Any $\lambda\in\spip(A)\setminus\spp(A)$ {\rm (}$\lambda\in\spim(A)\setminus\smm(A)${\rm )} is an eigenvalue of $A$ with a corresponding non-positive {\rm (}non-negative, respectively{\rm )} eigenvector.
\end{corollary}

Probably the most important
property of spectral points of type $\pi_+$ and type $\pi_-$ is their
 stability under compact perturbations.
  For this, let  $A$
be a closed, densely defined operator
in $\calH$. Recall that an operator
$K$  is said to be $A$-compact if $\dom A \subset \dom K$
and $K$, as a mapping from $(\calH_A,\hproduct_A)$ (cf.\ \eqref{PoeHoe}),
 into $\cal H$, is compact,
cf.\ \cite[IV \S 1.3]{k}.

\begin{theorem}\label{t:comp_1}
Let $A$ and $B$ be closed and densely defined operators in $\calH$, and assume that either $\dom B = \dom A$ such that $B - A$ is $A$-compact or that $\rho(A)\cap\rho(B)\neq\emptyset$ such that
\begin{equation} \label{Wittenberg2}
(A - \mu)^{-1} - (B- \mu)^{-1}
\end{equation}
is compact for some {\rm (}and hence for all\,{\rm )} $\mu\in\rho(A)\cap\rho(B)$. Then
\begin{align}
\begin{split}\label{Wittenberg3}
\spip(A)\cup\r(A) = \spip(B)\cup\r(B)\quad & \mbox{and}\quad
\spim(A)\cup\r(A) = \spim(B)\cup\r(B),\\
\infty\in\spp(A) &\Llra \infty\in\spp(B),\\
\infty\in\smm(A) &\Llra \infty\in\smm(B),\\
\infty\in\er(A)  &\Llra \infty\in\er(B).
\end{split}
\end{align}
\end{theorem}
\begin{proof}
Assume that the operator in \eqref{Wittenberg2} is compact. In the proof of \cite[Theorem 4.1]{abjt} $G$ is assumed to be boundedly invertible. This proof also holds for the current situation. Therefore, the first two equalities in
\eqref{Wittenberg3} follow from \cite[Theorem 4.1]{abjt}. Hence, it remains to prove the statements in \eqref{Wittenberg3} concerning $\infty$. We observe
\begin{equation*}
\begin{split}
\infty\in\er(A)   &\Llra A-\mu \text{ is a bounded operator} \Llra (A - \mu)^{-1} \text{ is a $\Phi_+$-operator}\\
                  &\Llra (B- \mu)^{-1} \text{ is a $\Phi_+$-operator}
                   \Llra \infty\in\er(B).
\end{split}
\end{equation*}
Let $\infty\in\spp(A)$ and let $(x_n)$ in $\dom(B)$ be a sequence with $x_n\to 0$, $Bx_n\wto 0$ as $n\to\infty$ and $\|Bx_n\|=1$. Denote by $K$ the operator
 in \eqref{Wittenberg2}, $K:=(A - \mu)^{-1} - (B- \mu)^{-1}$. Then
   $(B - \mu)x_n\wto 0$
implies $K(B - \mu)x_n\to 0$ as $n\to\infty$. Set
$$
u_n := (A - \mu)^{-1}(B - \mu)x_n = x_n - K(B - \mu)x_n.
$$
Then $u_n\to 0$ and since $Au_n = Bx_n + \mu(u_n - x_n)$ we have $Au_n\wto 0$, $\|Au_n\|\to 1$, $n\to\infty$,
$$
\liminf_{n\to\infty}\,[Bx_n,Bx_n] =
\liminf_{n\to\infty}\,[Au_n,Au_n],
$$
and $\infty\in\spp(B)$ follows from Theorem \ref{t:closed subspace} and Lemma \ref{l:infty_piplus}.

It remains to show \eqref{Wittenberg3} if $K := B-A$ is $A$-compact. By \cite[Theorem IV.1.11]{k} $K$ is $B$-compact. If $\lambda\in\r(A)$ then $A - \lambda$ is a $\Phi_+$-operator. By \cite[Theorem IV.5.26]{k} the same holds for $B - \lambda$, and we have $\lambda\in\spip(B)\cup\r(B)$ (see Remark \ref{r:PhiPlus}).

Let $\lambda\in\sap(B)\setminus\{\infty\}$ and let $(x_n)$ be a sequence in $\dom A$ with $\|x_n\|=1$, $x_n\wto 0$ and $(B - \lambda)x_n\to 0$ as $n\to\infty$. Since $(x_n)$ converges weakly to zero in $\calH_B$ we have $Kx_n\to 0$ and thus $(A - \lambda)x_n\to 0$ as $n\to\infty$. Hence, by Theorem \ref{t:closed subspace}, $\lambda\in\spip(A)$ implies $\lambda\in\spip(B)\cup\r(B)$.

For the point $\infty$ we have
$$
\infty\in\er(A) \Llra A \text{ bounded} \Llra B \text{ bounded}
\Llra \infty\in\er(B).
$$
Let $(x_n)$ be a sequence in $\dom A$ with $x_n\to 0$, $Bx_n\wto 0$ as $n\to\infty$ and $\|Bx_n\|=1$. As $(x_n)$ converges weakly to zero in $\calH_B$, it follows that $Kx_n\to 0$ as $n\to\infty$. Consequently, we have $\lim_{n\to\infty}\|Ax_n\| = 1$ and
$$
\liminf_{n\to\infty}\,[Bx_n,Bx_n] =
\liminf_{n\to\infty}\,[Ax_n,Ax_n].
$$
By Theorem \ref{t:closed subspace} and Lemma \ref{l:infty_piplus}, $\infty\in\spp(A)$ implies $\infty\in\spp(B)$.
\end{proof}

\section{Spectral Points of Type \texorpdfstring{$\pi_+$}{pi+} and \texorpdfstring{$\pi_-$}{pi-} of \texorpdfstring{$G$}{G}-Sym\-metric Operators}\label{binkrank1}
As in the previous section, let $G$ be a bounded selfadjoint operator in the Hilbert space $(\calH,\hproduct)$ inducing the inner product $\product = (G\cdot,\cdot)$. A linear operator $A$ in $\calH$ will be called {\it $G$-symmetric} (or $\product$-{\it symmetric}) if
$$
[Ax,y] = [x,Ay] \quad\text{ holds for all }x,y\in\dom A.
$$
Obviously, this is equivalent to $GA \subset (GA)^*$ where $^*$ denotes the adjoint with respect to the Hilbert space inner product $\hproduct$. If $GA = (GA)^*$ holds we say that $A$ is {\it $G$-selfadjoint}. E.g., such operators are studied in
\cite{adp,lmm} and in \cite{pt} in the case where the inner product $\product$ only has a finite number of non-positive squares.

In the sequel, let $A$ be a closed and densely defined $G$-symmet\-ric operator in $\calH$. A \emph{Jordan chain of $A$ at $\lambda\in\mathbb C$ of length} $n$ is a finite ordered set of non-zero vectors $\{x_0, \dots, x_{n-1}\}$ in $\dom A$ such that $(A-\lambda)x_0 = 0$  and $(A-\lambda)x_i = x_{i-1}$, $i = 1, \dots,n-1$. The vector $x_0$ is called the {\em eigenvector} of the Jordan chain. Several Jordan chains of $A$ at $\la$ are called linearly independent if their union is linearly independent. This holds if and only if the respective eigenvectors are linearly independent.
The \emph{algebraic eigenspace}  $\calL_\lambda(A)$ of $A$ at $\lambda$ is the collection of all Jordan chains of $A$ at $\lambda$,
$$
\calL_\lambda(A) := \bigcup_{n=1}^\infty\,\ker(A - \lambda)^n.
$$
The proof of \cite[Proposition 3.2]{dds} is also valid in the present situation (in \cite{dds} it is assumed that $G$ is boundedly invertible). Hence, we obtain for $\lambda,\mu\in\C$ with $\lambda\neq\ol\mu$
\begin{equation}\label{l:ker_perp_ker}
\calL_\lambda(A) \;\gperp\; \calL_\mu(A),
\mbox{ in particular, } \calL_\lambda(A) \mbox{ is neutral for non-real } \lambda.
\end{equation}

\subsection{Spectral points of type \texorpdfstring{$\pi_+$}{pi+} and \texorpdfstring{$\pi_-$}{pi-} and Jordan chains}
In the following lemma we collect some properties of spectral points of type $\pi_+$ and type $\pi_-$ and of spectral points of positive and negative type of $G$-symmetric operators.

\begin{lemma}\label{l:jordanchains}
Let $A$ be a closed, densely defined $G$-symmet\-ric operator in $\calH$. Then $\spp(A)$ and $\smm(A)$ are contained in $\ol\R$, and for $\lambda\in(\spip(A)\cup\spim(A))\setminus\ol\R$ the operator $A - \lambda$ is a $\Phi_+$-operator. For $\lambda\in\spip(A)\cup\spim(A)$ the following holds.
\begin{itemize}
\item[{\rm (i)}] The eigenvector of a Jordan chain of length greater than one of $A$ corresponding to $\lambda$  is an element of the isotropic part of $\ker(A - \lambda)$.
\item[{\rm (ii)}] $\ker(A - \lambda)$ is an Almost Pontryagin space, and there exists only finitely many linear independent Jordan chains of $A$ at $\lambda$.
\item[{\rm (iii)}] Let $\calN_+\,[\ds]\,\calN_-\,[\ds]\,\calN_0$ be a fundamental decomposition of $\ker(A - \lambda)$. Then there is an $A$-invariant linear manifold $\calL$ such that
$$
\calL_\lambda(A) = \calN_+\,[\ds]\,\calN_-\,[\ds]\,\calL.
$$
\end{itemize}
\end{lemma}
\begin{proof}
Let $\lambda\in\spip(A)\setminus\ol\R$. We show that  $A - \lambda$ is a $\Phi_+$-operator. Let $\calH_\lambda\subset\calH$ be a subspace of finite codimension as in Theorem \ref{t:closed subspace} (iv). Suppose that $A - \lambda$ is not a $\Phi_+$-operator.  Then, by \eqref{delta.delta}, there exists no $\veps > 0$ such that $\|(A - \lambda)x\| \ge \veps\|x\|$ for all $x\in\calH_\lambda\cap\dom A$. Hence there is a sequence $(x_n)$ in $\calH_\lambda\cap\dom A$ with $\|x_n\|=1$ and $(A - \lambda)x_n\to 0$ as $n\to\infty$. From $\Im\lambda \neq 0$ and $(-\Im\lambda)[x_n,x_n] = \Im[(A - \lambda)x_n,x_n] \to 0$ as $n\to\infty$ we conclude $\lim_{n\to\infty}[x_n,x_n] = 0$, a contradiction to $\lambda\in\spip(A)$, and $A - \lambda$ is a $\Phi_+$-operator. If, in addition, $\lambda\in\spp(A)\setminus\ol\R$ ($\lambda\in\smm(A)\setminus\ol\R$), then by Corollary \ref{c:spp_in_spip} and \eqref{l:ker_perp_ker}, we have $\ker(A - \lambda) = \{0\}$. Thus, as $A - \lambda$ is a $\Phi_+$-operator, $\lambda\in\r(A)$ follows, which is not possible. Therefore, $\spp(A)$ and $\smm(A)$ are contained in $\ol\R$.

We prove (i)--(iii). Let $\lambda\in\spip(A)\cup\spim(A)$. For $\lambda\notin\mathbb R$, (i) follows from \eqref{l:ker_perp_ker}. For $\lambda\in\R$, let $x$ be the eigenvector of a Jordan chain of $A$ at $\lambda$ of length greater than one. Then there exists $y\in\dom A$ such that $(A - \lambda)y = x$. Hence, for any $v\in\ker(A - \lambda)$ we have $[x,v] = [(A - \lambda)y,v] = [y,(A-\lambda)v] = 0$ and (i) is shown. By Lemma \ref{l:ap} $\ker(A - \lambda)$ is an Almost Pontryagin space. Hence, its isotropic part is finite-dimensional and (ii) follows from (i). Setting  $\calL := (\calN_+\,[\ds]\,\calN_-)^\gperp\cap\calL_\lambda(A)$ we obtain (iii).
\end{proof}

Let $[a,b]$ be an interval in $\R$ such that $[a,b]\cap\sap(A)\subset\spip(A)$. By Theorem \ref{t:comp_set} there exists an open neighborhood $\calU$ of $[a,b]$ in $\C$ such that also $\calU\cap\sap(A)\subset\spip(A)$. It is no restriction to assume that $\calU$ is connected. By Lemma \ref{l:jordanchains}, for every $\lambda\in\calU\setminus\R$ the operator $A - \lambda$ is a $\Phi_+$-operator. By \cite[IV \S 5.6]{k}, there exists a discrete set $\Xi\subset\calU\setminus\R$ such that $\dim\ker(A - \lambda)$ is constant on each of the two connected components of $\calU\setminus(\Xi\cup\R)$. The following theorem shows that in the special situation under consideration both these constants coincide and that it is possible to choose $\calU$ such that $\Xi = \emptyset$. The following theorem reveals an insight into the Jordan structures corresponding to the points in $\calU$. It extends \cite[Theorem 4.1]{bpt} to $G$-symmet\-ric operators. Moreover, the statements (a), (b), and (e) below are not contained in \cite{bpt}.

\begin{theorem}\label{t:big_one}
 Let $A$ be a closed, densely defined $G$-symmet\-ric operator in $\calH$.
Let $[a,b]$ be a compact interval in $\R$ with $[a,b]\cap\sap(A)\subset\spip(A)$ {\rm (}$[a,b]\cap\sap(A)\subset\spim(A)${\rm )}.
Then there exist an open neighborhood $\calU$ of $[a,b]$ in $\C$, a finite set $\sigma\subset [a,b]\cap(\spip(A)\setminus\spp(A))$ {\rm (}$\sigma\subset [a,b]\cap(\spim(A)\setminus\smm(A))$, respectively{\rm )} and a constant $\alpha \in \mathbb N$ such that for all $\lambda\in\calU\setminus\sigma$ we have
$$
\kappa_0(\ker(A - \lambda)) = \alpha \;\le\; \min_{\mu\in\sigma}\kappa_0(\ker(A - \mu))
$$
and
$$
\kappa_-(\ker(A - \lambda)) = 0 \quad
\bigl(\kappa_+(\ker(A - \lambda)) = 0,\text{ respectively}\,\bigr).
$$
If $\alpha = 0$, then
\begin{itemize}
\item[{\rm (a)}]  $\calL_\lambda(A)$,
$\lambda\in\calU$,
is an Almost Pontryagin space.
If $\lambda \in \calU\setminus\sigma$ then $\ker(A - \lambda)=\calL_\lambda(A)$
is a Hilbert space (anti-Hilbert space, respectively) and there is no
Jordan chain of length greater than one.
If $\lambda \in \sigma$ then every Jordan chain of $A$ at $\lambda$ is of finite length and, if
    $\calN_+^\lambda [\ds] \calN_-^\lambda [\ds] \calN_0^\lambda$ is a fundamental decomposition of $\ker(A - \lambda)$, then there exists a finite-dimensional
    subspace $\calL_1$ with
    $$
    \calL_\lambda(A) = \calN_+^\lambda [\ds] \calL_1 \quad
    \bigl(\calL_\lambda(A) = \calN_-^\lambda [\ds] \calL_1,\text{ respectively}\,\bigr).
    $$
\item[{\rm (b)}] The set $\sigma$ can be chosen as $\sigma = ([a,b]\cap\sap(A))\setminus\spp(A)$ {\rm (}$\sigma = ([a,b]\cap\sap(A))\setminus\smm(A)$, respectively{\rm )}.
\item[{\rm (c)}] $\calU\setminus\R\subset\r(A)$ and $\calU\cap\sap(A)\subset\spp(A)\cup\sigma$ $(resp.\ \calU\cap\sap(A)\subset\smm(A) \cup\sigma)$.
\end{itemize}
And in the case $\alpha > 0$ the following hold:
\begin{itemize}
\item[{\rm (d)}] $\calU\subset\sigma_{\rm p}(A)$ with $\big(\ker(A - \lambda)\big)^\iso\neq\{0\}$ for every $\lambda\in\calU$.
\item[{\rm (e)}] For each $\lambda\in\calU$ there exist at least $\alpha$ linearly independent Jordan chains of $A$ corresponding to $\lambda$ of infinite length.
\item[{\rm (f)}] $\calU\subset\spip(A)\setminus\spp(A)\quad$ {\rm (}$\calU\subset\spim(A)\setminus\smm(A)$, respectively{\rm )}.
\end{itemize}
\end{theorem}
\begin{proof}
We only prove the theorem for the case $[a,b]\cap\sap(A)\subset\spip(A)$. The statements follow from Lemma \ref{l:jordanchains}, Corollary \ref{c:spp_in_spip}, a compactness argument, and the following claim:\\[.2cm]
\noindent{\sc Claim. }{\it
Let $\lambda_0$ be a real point
with $\lambda_0 \in\spip(A)\cup\r(A)$. Then there exists an open neighborhood
 $\calU_0$ in $\C$ of $\lambda_0$ and an integer constant $\alpha_0$ such that the following holds:
\begin{enumerate}
\item[(i)]   For each $\lambda\in\calU_0\setminus\{\lambda_0\}$ we have
\begin{equation}\label{e:one-point}
\kappa_0(\ker(A - \lambda)) = \alpha_0 \le \kappa_0(\ker(A - \lambda_0))
\quad\text{and}\quad
\kappa_-(\ker(A - \lambda)) = 0.
\end{equation}
\item[(ii)]  If $\alpha_0 = 0$ then $\calL_{\lambda}(A)$ is an Almost Pontryagin space for all $\lambda\in\calU_0$ and
     $\calU_0\setminus \{\lambda_0\}\subset \spp(A)$.
     Every Jordan chain of $A$ corresponding to $\lambda$ is of finite length.
\item[(iii)] If $\alpha_0 > 0$ then for all $\lambda\in\calU_0$ there are at least $\alpha_0$ linearly independent Jordan chains of $A$ of infinite length.
\end{enumerate}
}

To prove the claim, we observe that by Theorem \ref{t:comp_set} we find an open neighborhood $\calV_0$ in $\C$ of $\lambda_0$ with $\calV_0\cap\sap(A)\subset\spip(A)$. For $\lambda\in\calV_0$ let $\ker(A - \lambda) = \calN_+^\lambda [\ds] \calN_-^\lambda [\ds] \calN_0^\lambda$ be a fundamental decomposition of $\ker(A - \lambda)$. As
 $\ker(A-\lambda)$ is an Almost Pontryagin space (cf.\ Lemma \ref{l:ap}),
 the spaces $\calN_-^\lambda$ and $\calN_0^\lambda$ are finite-dimensional.
 We define
\begin{equation*}
\begin{split}
\calN_0 &:= \linspan\{\calN_0^\lambda : \lambda\in\calV_0\setminus\{\lambda_0\},
\;\Im\lambda \ge 0\} \;\text{ and}\\
\calN_{-,0} &:= \linspan\{\calN_-^\lambda [\ds] \calN_0^\lambda :
\lambda\in\calV_0\setminus\{\lambda_0\}, \;\Im\lambda \ge 0\},
\end{split}
\end{equation*}
and set $\calL_0 := \ol{\calN_0}$ and $\calL_{-,0} := \ol{\calN_{-,0}}$. By \eqref{l:ker_perp_ker}, $\calL_0$ is neutral and $\calL_{-,0}$ is nonpositive. By $A_0$ ($A_{-,0}$) we denote the closure of the restriction of $A$ to $\calN_0$ ($\calN_{-,0}$, respectively) which then is a closed and densely defined $\product$-symmetric operator
in $\calL_0$ ($\calL_{-,0}$, respectively).
We will show that for $\lambda\in\calV_0\setminus\{\lambda_0\}$, $\Im\lambda \ge 0$,
\begin{eqnarray}\label{e:ker=N}
&\ker(A_0 - \lambda) = \calN_0^\lambda, \quad
\ker(A_{-,0} - \lambda) = \calN_-^\lambda\,[\ds]\,\calN_0^\lambda,&\\
\label{e:ker-sse-N}
& \ker(A_0 - \lambda_0) \subset \calN_0^{\lambda_0}.&
\end{eqnarray}
If $\lambda\in\calV_0\setminus\{\lambda_0\}$ with $\Im\lambda \ge 0$, the inclusion $\ker(A_0 - \lambda)\supseteq\calN_0^\lambda$ is obvious. Let $\lambda\in\calV_0$, $\Im\lambda\ge 0$, and $x\in\ker(A_0 - \lambda)$. Since $x\in\ker(A - \lambda)$ we find $x_+ \in\calN_+^\lambda$, $x_- \in\calN_-^\lambda$ and $x_0\in\calN_0^\lambda$ such that $x = x_+ + x_- + x_0$. Moreover, there exists a sequence $(x_n)$ in $\calN_0$ with $x_n\to x$ as $n\to\infty$. From $\calN_0 \gperp \calN_+^\lambda$ we conclude $[x_+,x_+] = [x,x_+] = \lim_{n\to\infty}[x_n,x_+] = 0$. Analogously, we obtain $[x_-,x_-] = 0$ and therefore $x = x_0 \in\calN_0^\lambda$. Thus, we have shown the first equation in \eqref{e:ker=N} and also the inclusion \eqref{e:ker-sse-N}. With a similar argument one shows the second equation in \eqref{e:ker=N}.

Since the operator $A - \lambda_0$ maps $\calN_0$ ($\calN_{-,0}$, respectively) surjectively onto itself, $\ran(A_0 - \lambda_0)$ ($\ran(A_{-,0} - \lambda_0)$, respectively) is dense in $\calL_0$ ($\calL_{-,0}$, respectively). If this range was
not closed, then $A_0 - \lambda_0$ ($A_{-,0} - \lambda_0$, respectively) would not be a $\Phi_+$-operator. Hence, by \eqref{delta.delta}, there exists an orthonormal sequence $(x_n)$ in $\dom A_0$ ($(x_n)$ in $\dom A_{-,0}$, respectively) with $(A - \lambda_0)x_n\to 0$ as $n\to\infty$. But this contradicts $\lambda_0\in\spip(A)$ since $\calL_0$ and $\calL_{-,0}$ are nonpositive subspaces, see Theorem \ref{t:closed subspace} (v). Thus, the operators $A_0 - \lambda_0$ and $A_{-,0} - \lambda_0$ are surjective Fredholm operators. By \cite[IV Theorem 5.22]{k} we find an open neighborhood $\calU_0\subset\calV_0$ in $\C$ of $\lambda_0$ such that all operators $A_0 - \lambda$ and $A_{-,0} - \lambda$ for $\lambda\in\calU_0$, $\Im\lambda\ge 0$, have the same index and are  surjective Fredholm operators. For all $\lambda\in\calU_0$ with $\Im\lambda\ge 0$ this gives
$$
\dim\,\ker(A_0 - \lambda) =  \dim\,\ker(A_0 - \lambda_0)\text{ and }
\dim\,\ker(A_{-,0} - \lambda)  = \dim\,\ker(A_{-,0} - \lambda_0).
$$
Set $\alpha_0 := \dim\,\ker(A_0 - \lambda_0)$ and
$\alpha_{-,0} := \dim\,\ker(A_{-,0} - \lambda_0)$.
Together with \eqref{e:ker=N} and \eqref{e:ker-sse-N}
we obtain  for all $\lambda\in\calU_0$ with $\Im\lambda\ge 0$
$$
\kappa_0(\ker(A - \lambda)) = \dim \calN_0^\lambda = \alpha_0
\leq  \dim \calN_0^{\lambda_0} = \kappa_0(\ker(A - \lambda_0)),
$$
which proves the first relation in \eqref{e:one-point} for $\Im\lambda\ge 0$. We have either $\alpha_0 = 0$ or $\alpha_0 > 0$ in which case for every $\lambda\in\calU_0$ the surjectivity of $A_0 - \lambda$ implies that any Jordan chain of $A_0$ corresponding to $\lambda$ is of infinite length. This shows (iii).
As $\calL_\lambda(A)$ is neutral for non-real $\lambda$, see \eqref{l:ker_perp_ker},
 $\calN_-^\lambda = \{0\}$ and we conclude with \eqref{e:ker=N} for
  $\Im\lambda > 0$
  that $\alpha_0$ and $\alpha_{-,0}$ coincide. But this, considering \eqref{e:ker=N} again, implies $\calN_-^\lambda = 0$ also for $\lambda\in\calU_0\cap\R$, $\lambda\neq\lambda_0$. Hence, the second relation in \eqref{e:one-point} holds for all $\lambda\in\calU_0\setminus\{\lambda_0\}$ with $\Im\lambda\ge 0$. Applying similar arguments as above for the lower complex plane we obtain \eqref{e:one-point} for all $\lambda\in\calU_0\setminus\{\lambda_0\}$.
%

It remains to prove (ii). Due to \eqref{l:ker_perp_ker}
 and \eqref{e:one-point} we have
\begin{equation}\label{e:makesnicer}
\calL_\lambda(A) = \{0\}\quad\text{for }\lambda\in\calU_0\setminus\R.
\end{equation}
\eqref{e:one-point} and Lemma \ref{l:jordanchains} also imply that
$$
\calL_\lambda(A) = \ker(A - \lambda) = \calN_+^\lambda\quad\text{for }\lambda\in\calU_0\cap\R,\;\lambda\neq\lambda_0.
$$
With Corollary \ref{c:spp_in_spip}, we obtain
$\calU_0\setminus\{\lambda_0\}\subset \spp(A)$ and $\calL_\lambda(A)$
is uniformly positive and, hence, also an Almost Pontryagin space. Therefore
 we have to show (ii) only for $\calL_{\lambda_0}(A)$. By Lemma \ref{l:jordanchains} it suffices to show that each Jordan chain of $A$ corresponding to $\lambda_0$ is finite. Assume the contrary. Then there exists an infinite sequence $(x_n)$ in $\dom A$ such that $(A - \lambda_0)x_{k+1} = x_k$ for all $k\ge 0$ and $(A - \lambda_0)x_0 = 0$. Since
$$
[x_n,x_m] = [(A - \lambda_0)^m x_{n+m},x_m] = [x_{n+m},(A - \lambda_0)^m x_m] = 0
$$
holds for all $n,m\in\N$, $\calM_0 := \linspan\{x_n : n\in\N\}$ is neutral, and $A - \lambda_0$ maps $\calM_0$ surjectively onto itself. Let $A_1$ be the closure of $A|\calM_0$ in the neutral subspace $\calM_1 := \overline{\calM_0}$. As above for $A_0 - \lambda_0$ it can be shown that $A_1 - \lambda_0$ is a surjective Fredholm operator. Since $x_0\in\ker(A_1 - \lambda_0)$ it follows
\cite[IV Theorem 5.22]{k} that
$\ker(A - \lambda)\supset\ker(A_1 - \lambda)\neq\{0\}$ for all $\lambda$ in a
neighborhood of $\lambda_0$, which contradicts \eqref{e:makesnicer}.
\end{proof}

 We refer to \cite[Section 4]{pt} for an example with $\alpha_0>0$ where the unit disc is contained in $\spip(A)\setminus \spp(A)$.


\subsection{Finite rank perturbations}\label{ss:frp}
In this section we construct a finite rank perturbation which turns a real spectral point of type $\pi_+$ (type $\pi_-$) into a spectral point of positive (negative, respectively) type. It was shown in \cite{jl} that such a finite rank perturbation exists in the case where $A$ is a definitizable operator in a Krein space. The proof of the  following lemma is omitted as the proof of
\cite[Lemma 3.10]{pt} proves also the statements of Lemma \ref{l:APS} below.

\begin{lemma}\label{l:APS}
Let $\calD$ be a dense linear manifold in $\calH$ and let $\calL\subset\calD$ be an Almost Pontryagin space. If $\calL = \calL_+ [\ds] \calL_- [\ds] \calL^\iso$ is a fundamental decomposition of $\calL$, then there exist subspaces $\calL_{00},\calL_{01},\calP\subset\calD$ and $\calM\subset\calH$ such that
\begin{equation*}\label{e:APS-DEC}
\calH = \calL_+ [\ds] \calL_- [\ds] \calL_{00} [\ds]
        (\calL_{01} \ds \calP) [\ds] \calM,
\end{equation*}
and the following statements hold
\begin{enumerate}
\item[{\rm (i)}]   $\calL_{00} = \calL^\iso \cap \calH^\iso$
               and $\calL^\iso = \calL_{00} \ds \calL_{01}$,
\item[{\rm (ii)}]  $\calP$ is neutral,
\item[{\rm (iii)}] $\calP\cap\calL_{01} = \calL_{01}\cap\calP^\gperp =
                    \calL_{01}^\gperp\cap\calP = \{0\}$,
\item[{\rm (iv)}]  $\calG := \calL_{01} \ds \calP$ is non-degenerate,
that is $\calG \cap \calG^\iso=\{0\}$,
\item[{\rm (v)}]   $\kappa_+(\calG) = \kappa_-(\calG) = \dim\calP =
                    \dim\calL_{01} < \infty$,
\item[{\rm (vi)}]  $\calL^\gperp = \calL^\iso [\ds] \calM$.
\end{enumerate}
Moreover, there exists a fundamental symmetry $J$ in $\calG$ such that $\calP = J\calL_{01}$.
\end{lemma}

We  now state the above-mentioned theorem.
We refer to \cite{bpt} where a similar result is shown for
(the special case of) self adjoint
operators in Krein spaces. Contrary to \cite{bpt}, the proof of
Theorem \ref{t:finite_rank} below is based on Lemma \ref{l:APS}.

\begin{theorem}\label{t:finite_rank}
 Let $A$ be a closed, densely defined $G$-symmet\-ric operator in $\calH$ and
let $0\notin\sigma_{\rm p}(G)$. If \,$\lambda\in\spip(A)\cap\R$ {\rm (}$\lambda\in\spim(A)\cap\R${\rm )} then there is a $G$-symmetric and bounded finite rank operator $F$ such that
$$
\lambda\in\spp(A + F)\cup\r(A + F) \quad
(\lambda\in\smm(A + F)\cup\r(A + F),\text{ respectively})
$$
and
$$
\dim\ran(F) = \kappa_{-,0}(\ker(A - \lambda))\quad
\bigl(\dim\ran(F) = \kappa_{+,0}(\ker(A - \lambda)),\text{ respectively}\,\bigr).
$$
\end{theorem}
\begin{proof}
$\ker(A - \lambda)$ is an Almost Pontryagin space (Lemma \ref{l:ap}) with
  a fundamental decomposition
  $\calL_+ [\ds] \calL_- [\ds] \calL_0$.
   There are $\calP\subset\dom A$ and $\calM\subset\calH$ as in Lemma \ref{l:APS} ($\calL_{00} = \{0\}$) with
$$
\calH = \calL_+ [\ds] \calL_- [\ds] (\calL_{01} \ds \calP) [\ds] \calM.
$$
The subspace $\calL_+$ is uniformly positive and, by \eqref{Elgersburg},
$\calH = \calL_+[\ds]\calL_+^\gperp$. Hence a bounded projection
onto  $\calL_+^\gperp =
\calL_- [\ds] (\calL_{01} \ds \calP) [\ds] \calM$ exists. The subspaces
$\calL_{01}$ and $\calP$ are finite-dimensional, therefore $(\calL_{01} \ds \calP) [\ds] \calM$ is also closed and the bounded projection $P_-$ onto $\calL_-$
exists. A similar reasoning shows that there exists
the bounded projection $P_0$ onto  $\calL_{01}$.
 With the fundamental symmetry $J$ in $\calL_{01} \ds \calP$ with $J\calL_{01} = \calP$ (see Lemma \ref{l:APS}) define
$$
F := P_- + JP_0.
$$
This operator is $G$-symmetric since $P_-$ is $G$-symmetric and
$$
[JP_0x,y] = [JP_0x,P_0y] = [P_0x,JP_0y] = [x,JP_0y]
\;\text{ for }\;x,y\in\calH.
$$
Assume $\lambda\in\sap(A + F)$. Set
$\calH_\lambda := \calL_+ [\ds] \calP [\ds] \calM$. By $\lambda \in \spip(A)$
and Theorem \ref{t:H0}, $\calH_\lambda$ is a subspace as
as in Definition \ref{d:spip} (for $A$).
 Since $A|\calH_\lambda = (A + F)|\calH_\lambda$ we have
 $\lambda\in\spip(A + F)$. Moreover, the inclusion
  $\calL_+ \subset \ker(A + F - \lambda)$ holds. If, conversely,
  $x\in\ker(A + F - \lambda)$, then we have $P_-x$, $P_0x \in
  \ker(A - \lambda)$,
  $[x,(A-\lambda)P_- x]=0$,
  $[JP_0x,P_-x]=0$, $[x,(A-\lambda)P_0x]=0$,
  $[P_- x,P_0x]=0$, and
\begin{equation*}
\begin{split}
[P_-x,P_-x] &= [P_-x,P_-x] + [(A - \lambda)x,P_-x] + [JP_0x,P_-x]\\
            &= [(A + F - \lambda)x,P_-x] = 0 \quad \mbox{and}\\
[JP_0x,P_0x] &= [JP_0x,P_0x] + [(A - \lambda)x,P_0x] + [P_-x,P_0x]\\
             &= [(A + F - \lambda)x,P_0x] = 0.
\end{split}
\end{equation*}
Hence we have  $P_-x=P_0x=0$, $Fx = 0$,  $x\in\ker(A - \lambda)$, and thus $x\in\calL_+$. With Corollary \ref{c:spp_in_spip}, $\lambda\in\spp(A + F)$ follows.
\end{proof}

Theorem \ref{t:finite_rank} and Theorem \ref{t:big_one} together with Corollary \ref{c:spp_in_spip} yield

\begin{corollary}
 Let $A$ be a closed, densely defined $G$-symmet\-ric operator in $\calH$,
let $0\notin\sigma_{\rm p}(G)$, let  $[a,b]\cap\sap(A)\subset\spip(A)$ or $[a,b]\cap\sap(A)\subset\spim(A)$.
Then there exists a $G$-symmetric, bounded finite rank operator $F$ and an open neighborhood $\calU$ of $[a,b]$ in $\C$ such that
$\calU\setminus\R\subset\r(A + F)$.
\end{corollary}


\subsection{Growth of the Resolvent}\label{ss:growth}
The second part of the following Theorem was proved in \cite{lmm} for bounded $G$-symmetric operators.

\begin{theorem}\label{t:langer}
 Let $A$ be a closed, densely defined $G$-symmet\-ric operator in $\calH$.
\begin{itemize}
\item[{\rm (a)}] If $[a,b]\cap\sap(A)\subset\spip(A)$ or $[a,b]\cap\sap(A)\subset\spim(A)$, then there exist an open neighborhood $\calU$ in $\C$ of $[a,b]$, a subspace $\calH_0\subset\calH$ with $\codim\calH_0 < \infty$ and a number $c > 0$ such that
\begin{equation}\label{e:PiPlusPhiPlus}
\|(A - \lambda)x\| \ge c|\Im\lambda|\|x\|
\end{equation}
holds for all $x\in\calH_0\cap\dom A$ and all $\lambda\in\calU\setminus\R$.

\item[{\rm (b)}] If $[a,b]\cap\sap(A)\subset\spp(A)$ or $[a,b]\cap\sap(A)\subset\smm(A)$, then there are $\calU$ and $c$ as in {\rm (a)} such that {\rm \eqref{e:PiPlusPhiPlus}} holds for all $x\in\dom A$ and all $\lambda\in\calU\setminus\R$. In particular, $\calU\setminus\R\subset\r(A)$, and if even $\calU\setminus\R\subset\rho(A)$ holds, then with $M := c^{-1}$ we have
\begin{equation}\label{e:res_pp}
\|(A - \lambda)^{-1}\| \le \frac{M}{|\Im\lambda|} \;\;\text{ for all }
\lambda\in\calU\setminus\R.
\end{equation}
If $G$ is boundedly invertible and the operator $GA$ is selfadjoint in $\calH$, then $\calU$ can be chosen such that $\calU\setminus\R\subset\rho(A)$.
\end{itemize}
\end{theorem}
\begin{proof}
(a). Let us assume that $[a,b]\cap\sap(A)\subset\spip(A)$. Set $K := [a,b]$, let $\calU$, $\calH_0$, and $\veps$ be as in Theorem \ref{t:comp_set} and define $c:=\min\{\veps/\|G\|,1\}$. It is no restriction to assume $|\Im\lambda| < \veps$ for all $\lambda\in\calU$. Now, let $x\in\calH_0\cap\dom A$ and $\lambda\in\calU\setminus\R$. If $\|(A - \lambda)x\|\ge \veps\|x\|$, relation \eqref{e:PiPlusPhiPlus} clearly holds and if $\|(A - \lambda)x\|\le \veps\|x\|$, then by Theorem \ref{t:comp_set} we have $[x,x]\ge\veps\|x\|^2$, and
$$
|\Im\lambda|\veps\|x\|^2 \le |\Im[\lambda x,x]| =
|\Im[(A - \lambda)x,x]| \le \|G\|\|(A - \lambda)x\|\|x\|
$$
follows. This shows (a).

(b). For the proof of the first part of (b) apply Theorem \ref{t:comp_set_pp} instead of Theorem \ref{t:comp_set} in the argumentation above and we obtain that the inequality \eqref{e:PiPlusPhiPlus} is valid for all $x\in\dom A$, thus $\calU\setminus\R\subset\r(A)$ holds. Hence, \eqref{e:res_pp} follows from the assumption $\calU\setminus\R\subset\rho(A)$.
If $G$ is boundedly invertible and $GA = (GA)^*$ then $A$ is selfadjoint in the Krein space $(\calH,\product)$. Choose $\calU$ as in the first part of (b) such that $\calU$ is symmetric with respect to $\R$. Then $\calU\setminus\R\subset\r(A)$ and  for $\lambda\in\calU\setminus\R$ we have $\ker(A - \lambda) = \ker(A - \ol\lambda) = \{0\}$ and $\ran(A - \lambda) = \ol{\ran(A - \lambda)} = \ker(A - \ol\lambda)^\gperp = \calH$. This proves $\lambda\in\rho(A)$.
\end{proof}

The following theorem shows that an inequality similar to \eqref{e:res_pp} holds in a neighborhood of intervals with spectral points of type $\pi_+$ or regular points of $A$.

\begin{theorem}\label{t:resolvente}
Let $A$ be a closed, densely defined $G$-symmet\-ric operator in $\calH$.
Let $[a,b]\cap\sap(A)\subset\spip(A)$ or $[a,b]\cap\sap(A)\subset\spim(A)$
such that there is an open neighborhood $\calU$ of $[a,b]$ in $\C$ with
$\calU\setminus\R\subset\rho(A)$. Then there exists an open neighborhood $\calV$ of
$[a,b]$ in $\C$ and constants $M>0$ and $m\in\N$ such that for $\lambda\in\calV\setminus\R$ we have
\begin{equation}\label{e:resolvente}
\|(A - \lambda)^{-1}\| \le \frac{M}{|\Im\,\lambda|^m}.
\end{equation}
\end{theorem}
\begin{proof}
In view of Theorems \ref{t:langer} and \ref{t:big_one} it is sufficient to prove Theorem \ref{t:resolvente} in a
 neighbourhood of a spectral point $\lambda\in\spip(A)\setminus \spp(A)$.
   Choose a fundamental decomposition of $\ker(A - \lambda) = \calN_+ [\ds] \calN_- [\ds] \calN_0$. By Lemma \ref{l:ap}, $\calN_- [\ds] \calN_0$ is finite-dimensional and by Corollary \ref{Steinach}
   it contains at least one non-zero element.
   We set $A_0 := A|\calN_+^\gperp$.
   Then $A_0$ is closed, densely defined in $\calN_+^\gperp$,
   $G$-symmetric with $\lambda\in \spip(A_0)$ and $\calU\setminus\R\subset\rho(A_0)$.
  We have
  $\calL_\lambda(A) = \calN_+ [\ds]  \calL_\lambda(A_0)$
   and by Theorem \ref{t:big_one} (b) the subspace $\calL_\lambda(A_0)$ is finite-dimensional. If $\calL_\lambda(A_0) = \calL_+ [\ds] \calL_- [\ds] \calL_0$ is a
    fundamental decomposition of $\calL_\lambda(A_0)$,
    then $\calL_+$, $\calL_-$ and $\calL_0$ are finite-dimensional. With Lemma \ref{l:APS} applied to $\calL_\lambda(A_0)$
we find subspaces $\calL_{00},\calL_{01},\calP\subset\dom(A_0)$ and $\calM\subset\calN_+^\gperp$ which satisfy $\calL_0 = \calL_{00}[\ds]\calL_{01}$ and
$$
\calN_+^\gperp = \calL_+ [\ds] \calL_- [\ds] \calL_{00} [\ds] (\calL_{01} \ds \calP) [\ds] \calM.
$$
Set $\calL_1 := \calL_+ [\ds] \calL_-$. With \eqref{Elgersburg} we obtain
\begin{equation}\label{e:res_dec}
\calH = \calN_+ \ds \calL_0 \ds \calL_1 \ds \calM \ds \calP.
\end{equation}
With respect to the decomposition \eqref{e:res_dec}
 the operator $A$ can be represented as
\begin{equation}\label{e:matrix_representation}
A =
\begin{pmatrix}
\lambda &    0   &    0   &    0   &    0   \\
 0  & A_{11} & A_{12} & A_{13} & A_{14} \\
 0  &    0   & A_{22} &    0   & A_{24} \\
 0  &    0   &    0   & A_{33} & A_{34} \\
 0  &    0   &    0   &    0   & A_{44}
\end{pmatrix}.
\end{equation}
This is a consequence of the fact that $\calN_+$, $\calN_+^\gperp$
$\calL_0$, $\calL_1\ds\calL_0$ as well as $\calL_0^\gperp = \calL_0 \ds \calM$ are $A$-invariant. On $\dom A$ define the operators
\begin{equation}\label{GerabergimFebruar}
K :=
\begin{pmatrix}
 0  &    0    &    0    &    0    &    0    \\
 0  & -A_{11} & -A_{12} & -A_{13} & -A_{14} \\
 0  &    0    & -A_{22} &    0    & -A_{24} \\
 0  &    0    &    0    &    0    & -A_{34} \\
 0  &    0    &    0    &    0    & -A_{44}
\end{pmatrix}
\quad\text{and}\quad
\wt A :=
\begin{pmatrix}
\lambda &    0   &    0   &    0   &    0   \\
 0  &    0   &    0   &    0   &    0   \\
 0  &    0   &    0   &    0   &    0   \\
 0  &    0   &    0   & A_{33} &    0   \\
 0  &    0   &    0   &    0   &    0
\end{pmatrix}.
\end{equation}
Then, $\wt A = A + K$, and $K$ is easily seen to be $A$-bounded,
 that is, $\wt A$ is a bounded mapping from
 $(\calH_A, \hproduct_A)$ (cf.\ \eqref{PoeHoe})
  into $\calH$, see \cite[IV \S 1.1]{k}.
 As $K$ has a finite-dimensional range, $K$ is $A$-compact. By \cite[Theorem IV.1.11]{k}, the operator $\wt A$ is closed. Hence also $A_{33}$ is closed. We have
\begin{equation}\label{GerabergimMaerz}
\rho(A)\setminus\{\lambda\}  = \rho(A_{11})\cap\rho(A_{22})\cap\rho(A_{33})\cap\rho(A_{44})\setminus\{\lambda\}.
\end{equation}
Moreover, $A_{33}$ is $\product$-symmetric and if
for some $x\in\calM\cap\dom A$ we have $(A_{33}-\la)x=0$,
then \eqref{e:matrix_representation} implies $(A-\la)x\in\calL_\la(A)$,
hence $x\in\calL_\la(A)$ which yields $x=0$. Thus, we have
\begin{equation}\label{Hbf}
\ker(A_{33} - \la) = \{0\}.
\end{equation}
In addition, the following holds:
\begin{equation}\label{zuhause}
\la\in\spp(A_{33} )\cup\rho(A_{33}).
\end{equation}
To prove \eqref{zuhause}, assume $\la\in\r(A_{33})$. Then we find a neighbourhood
(cf.\ \eqref{e:rA}) $\calW$
of $\lambda$ in $\mathbb C$ such that $\calW \subset\r(A_{33})$ and
$\calW \subset \calU$. By assumption $\calW \setminus\mathbb R\subset
\calU\setminus\mathbb R\subset \rho(A)$ and we conclude from
\eqref{GerabergimMaerz} that $A_{33}-\lambda$ is a Fredholm operator
of index zero which together with \eqref{Hbf} implies $\lambda\in
\rho(A_{33})$. It remains to consider the case $\lambda \in \sap(A_{33})$.
By \eqref{GerabergimFebruar}, $\lambda \in \sap(\wt A)$ and, as $\lambda \in
\spip(A)$ we conclude from Theorem \ref{t:comp_1} that
$\lambda\in\spip(\wt A)$. Also, by  \eqref{GerabergimFebruar},
$\lambda\in\spip(A_{33})$. In view of Corollary \ref{Steinach} and \eqref{Hbf} we obtain
$\lambda\in\spp(A_{33})$ and \eqref{zuhause} is proved.

Taking into account that the operators $A_{11}$, $A_{22}$ and $A_{44}$ act in spaces of finite dimension and using Theorem \ref{t:langer} (b) we find an open neighborhood $\calV$ of $\lambda$ in $\C$ and constants $M_1,M_2 > 0$ as well as $m_1\in\N$ such that
$$
\max_{k = 1,2,4}\|(A_{kk} - \mu)^{-1}\| \le \frac{M_1}{|\mu - \lambda|^{m_1}}
\quad\text{and}\quad
\|(A_{33} - \mu)^{-1}\| \le \frac{M_2}{|\Im\,\mu|}
$$
for all $\mu\in\calV\setminus\R$. Now, by using \eqref{e:matrix_representation} it is easily seen that \eqref{e:resolvente} holds.
\end{proof}



\section{The Local Spectral Function}\label{binkrank2}
Let $(a,b)$ be a real open interval with $-\infty\le a < b\le\infty$. By $\mathscr{M}(a,b)$ we denote the set consisting of all bounded intervals $\Delta$ whose closure is contained in $(a,b)$ and finite unions of such intervals. If $S$ is a discrete subset of $(a,b)$, we set
$$
\mathscr{M}_S(a,b) := \{\Delta\in\mathscr{M}(a,b) : \partial\Delta\cap S = \emptyset\}.
$$
Note that $S$ may accumulate to $a$ or $b$ and
that $\mathscr{M}_\emptyset(a,b) = \mathscr{M}(a,b)$.
We shall say that the bounded operator $B$ {\it commutes with $A$} if $BA\subset AB$, i.e.\
$$
x\in\dom A \;\Lra\; Bx\in\dom A \,\text{ and }\, ABx = BAx.
$$
If $\rho(A)\neq\emptyset$ this is equivalent to the fact that $B$ commutes with the resolvent of $A$.

\begin{definition}\label{d:spectral_function}
Let $S$ be a discrete subset of the open {\rm (}and maybe unbounded{\rm )} interval $(a,b)$. A mapping $E$ from $\mathscr{M}_S(a,b)$ into the set of bounded projections on $\calH$ is called a {\it local spectral function for $A$ on $(a,b)$} 
if $E(\emptyset)=0$ and the following conditions are satisfied:
\begin{enumerate}
\item[{\rm (S1)}] $E(\Delta_1\cap\Delta_2) = E(\Delta_1)E(\Delta_2)$ for all $\Delta_1,\Delta_2\in\mathscr{M}_S(a,b)$.
\item[{\rm (S2)}] If $\Delta_1,\Delta_2,\ldots\in\mathscr{M}_S(a,b)$ are mutually disjoint and $\Delta := \bigcup_{i=1}^\infty\Delta_i\in\mathscr{M}_S(a,b)$ then
$$
E(\Delta) = \sum_{i=1}^\infty E(\Delta_i),
$$
where the sum converges in the strong operator topology.
\item[{\rm (S3)}] If the bounded operator $B$ commutes with $A$, then it commutes with every $E(\Delta)$, $\Delta\in\mathscr{M}_S(a,b)$.
\item[{\rm (S4)}] $\sigma(A|E(\Delta)\calH)\subset\ol\Delta$.
\item[{\rm (S5)}] $\sigma(A|(I - E(\Delta))\calH)\subset\ol{\sigma(A)\setminus\Delta}$.
\end{enumerate}
\end{definition}


For a bounded operator $A$ with $(a,b)\cap\sap(A)\subset\spp(A)$ and $\calU\setminus\R\subset\rho(A)$ for some open neighborhood $\calU\subset\C$ of $(a,b)$ it is proved in \cite{lmm} that there exists a set function $E$ defined on $\mathscr{M}(a,b)$ with  (S1), (S2), (S4), and (S5) as in Definition \ref{d:spectral_function}. Moreover, every $E(\Delta)$ is a $G$-symmetric projection onto a uniformly positive subspace. Thus, $(E(\Delta)\calH, \product)$ is a Hilbert space and the restriction of $A$ to $E(\Delta)\calH$ is a selfadjoint operator.

It is not mentioned in \cite{lmm} that $E$ also has the property (S3). We will show this in the next theorem. Moreover, we extend the results on the local spectral function from \cite[Section 3]{lmm} to unbounded operators. 
We mention that Theorem \ref{t:B} below is contained in \cite[Theorem 2.7]{adp}.  However, in \cite{adp} the property (S3) is not proved explicitly.
Therefore we prefer to give a detailed proof here.

\begin{theorem}\label{t:B}
Let $A$ be a closed and densely defined $G$-symmet\-ric operator in $\calH$. Assume that $[a,b]\cap\sap(A)\subset\spp(A)$, and let $\calU$ be an open neighborhood of $[a,b]$ in $\C$ with $\calU\setminus\R\subset\rho(A)$. Then there exists a local spectral function for $A$ defined on $\mathscr{M}(a,b)$ as in Definition \ref{d:spectral_function}. Moreover, every $E(\Delta)$, $\Delta\in\mathscr{M}(a,b)$, is a $G$-symmetric projection onto a uniformly positive subspace.
\end{theorem}
\begin{proof}
{\bf 1. }By Theorem \ref{t:langer}(b) and
\cite[Chapter II, \paragraf 2, Theorem 5]{lm}, the maximal spectral subspace $\calL_{[a,b]}$ of $A$ corresponding to $[a,b]$ exists. Recall that the maximal spectral subspace $\calL_\Delta$ of $A$ corresponding
to a compact interval $\Delta$ (if it exists) has the following properties
(cf.\ \cite[Chapter I, \paragraf 4]{lm}):
\begin{itemize}
\item[I.]  $\calL_\Delta\subset\dom A$ is $A$-invariant.
\item[II.] $\sigma(A|\calL_\Delta)\subset\Delta\cap\sigma(A)$.
\item[III.] If $\calL\subset\dom A$ is an $A$-invariant subspace and $\sigma(A|\calL)\subset\Delta$ then $\calL\subset\calL_\Delta$.
\end{itemize}
We set $A_1:= A|\calL_{[a,b]}$. Then it follows from $\sigma(A_1)\subset [a,b]$ and \eqref{e:partial} that $\sigma(A_1) = \sap(A_1)$. This and the assumption $[a,b]\cap\sap(A)\subset\spp(A)$ yield $\sigma(A_1) = \spp(A_1)$.
By \cite{lmm}, $(\calL_{[a,b]},\product)$ is a Hilbert space, and the restriction of $A$ to $\calL_{[a,b]}$ is selfadjoint in this Hilbert space. Consequently, it has a spectral function which we denote by $E_1$. As $A$ is $G$-symmet\-ric, also $\calL_{[a,b]}^{[\perp]}$ is $A$-invariant.
Let $P$ be the projection onto $\calL_{[a,b]}$ with respect to the decomposition $\calH = \calL_{[a,b]}[\ds]\calL_{[a,b]}^\gperp$. Then, we define $E$ by
\begin{equation*}
E(\Delta) := E_1(\Delta)P,\quad\Delta\in\mathscr{M}(a,b).
\end{equation*}
It is not difficult to see that $E$ satisfies (S1), (S2), (S4) and (S5).

{\bf 2. }\;
It remains to show (S3). As $E_1$ is a spectral function of a selfadjoint
operator in a Hilbert space, it is sufficient
 to prove (S3) for a compact interval $\Delta_0 \subset (a,b)$.
 Note that $\calL_{\Delta_0} := E(\Delta_0)\calH$ is the maximal spectral subspace of $A$
 corresponding to $\Delta_0$.
 Let $B$ be a bounded operator which commutes with $A$. In order to
 show that $B$ commutes with $E(\Delta_0)$ it is sufficient to show
 that $\calL_{\Delta_0}$ and $\calL_{\Delta_0}^{[\perp]}$ are $B$-invariant.
 It is easily checked that the proof of  \cite[Proposition 1.3.2]{cf} is also
  valid for an unbounded operator $A$ and we obtain $B\calL_{\Delta_0}
  \subset \calL_{\Delta_0}$. It remains to prove
\begin{equation}\label{Trinken}
B\calL_{\Delta_0}^\gperp\subset\calL_{\Delta_0}^\gperp.
\end{equation}
To see this, let $\Delta\subset (a,b)$ be a compact interval such that $\Delta_0$ is contained in the interior of $\Delta$. We will show
\begin{equation}\label{Essen}
B\calL_\Delta^\gperp\subset\calL_{\Delta_0}^\gperp.
\end{equation}
Then \eqref{Trinken} follows from the fact that
for $\alpha,\beta\in (a,b)$, $\alpha < \beta$,
\begin{equation*}\label{e:schnitt}
\bigcap_{\veps > 0}\,\calL_{[\alpha - \veps,\beta + \veps]} = \calL_{\Delta},
\end{equation*}
which easily follows from the properties of maximal spectral subspaces.

In order to show \eqref{Essen}, let $\lambda_0\in\rho(B)$ and set $\calK := (B - \lambda_0)\calL_\Delta^\gperp$. Evidently, $\calK$ is closed. Since $\calL_\Delta^\gperp$ is $A$-invariant and $BA\subset AB$ holds, $\calK$ is $A$-invariant. From
$$
A|\calL_\Delta^\gperp = \Big[\,(B - \lambda_0)^{-1}|\calK\,\Big] \; \Big[\,A|\calK\,\Big] \; \Big[\,(B - \lambda_0)|\calL_\Delta^\gperp\,\Big]
$$
and (S5) we conclude that
\begin{equation}\label{e:K_and_L}
\sigma(A|\calK) = \sigma\left(A|\calL_\Delta^\gperp\right) \subset \ol{\sigma(A)\setminus\Delta}.
\end{equation}
Let $x\in\calK$, $x = u+v$, where $u\in\calL_\Delta$ and $v\in\calL_\Delta^\gperp$. For $\lambda\in\rho(A)$ we have
$$
(A - \lambda)^{-1}u = (A|\calK - \lambda)^{-1}x - (A|\calL_\Delta^\gperp - \lambda)^{-1}v,
$$
and from \eqref{e:K_and_L} it follows that this function admits a holomorphic continuation to $(a',b')$ where $\Delta = [a',b']$. As $(A - \lambda)^{-1}u\in\calL_\Delta$ for $\lambda\in\rho(A)$ and $\sigma(A|\calL_\Delta)\subset\Delta$, the function $\lambda\mapsto (A|\calL_\Delta - \lambda)^{-1}u$ extends to a holomorphic function $\C\setminus\{a',b'\}\to\calL_\Delta$. Since $(\calL_\Delta,\product)$ is a Hilbert space and $A|\calL_\Delta$ is selfadjoint in this Hilbert space,
$$
u\in\ker(A - a')\ds\ker(A - b')\subset\calL_{\Delta_0}^\gperp,
$$
follows and, hence,
 $x = u + v \in \calL_{\Delta_0}^\gperp + \calL_\Delta^\gperp \subset \calL_{\Delta_0}^\gperp$.
\end{proof}

 The next theorem is the main result in this section.
\begin{theorem}\label{t:sf}
 Let $A$ be a closed, densely defined $G$-symmet\-ric operator in $\calH$ and
$(a,b)$ a {\rm (}possibly unbounded{\rm )} open interval in $\mathbb R$ with $(a,b)\cap\sap(A)\subset\spip(A)$ such that there exists an open neighborhood $\calU$ of $(a,b)$ in $\C$ with $\calU\setminus\R\subset\rho(A)$. Then $A$ has
 a local spectral function $E$ on $(a,b)$ with  $S := (a,b)\cap(\spip(A)\setminus\spp(A))$. For $\Delta\in\mathscr{M}_S(a,b)$
  the projection $E(\Delta)$ is $G$-selfadjoint, and its range is an Almost Pontryagin space with finite rank of non-positivity.
\end{theorem}
\begin{proof}
The proof is divided into three parts. First, we define the spectral function, then we prove that it satisfies (S1)--(S5) and show in the last part that the spectral projections map to Almost Pontryagin spaces.

\
\\
\indent{\bf 1. }Let $\Delta\in\mathscr M_S(a,b)$ be an interval with endpoints $a'$ and $b'$, $a' < b'$. We choose numbers $a'',b''\in (a',b')$, $a'' < b''$, such that $[a',a'']\cup [b'',b']$ has no common point with $S$ (cf.\ Theorem \ref{t:big_one}). Then $\Delta_1 := \Delta\cap[a',a'']$ and $\Delta_2 := \Delta\cap [b'',b']$ are of positive type with respect to $A$. By Theorem \ref{t:B}, $A$ has a spectral function $E_j$ on $\ol{\Delta_j}$, $j=1,2$, such that the spectral subspace
\begin{equation*}\label{e:spaces}
\calL_j := E_j(\Delta_j)\calH, \quad j=1,2,
\end{equation*}
of $A$ is uniformly positive. Moreover, as $\sigma(A|\calL_1\cap\calL_2) = \emptyset$, we have $\calL_1\cap\calL_2 = \{0\}$ and thus, by (S3), $E_1(\Delta_1)E_2(\Delta_2) = 0$ which implies $\calL_1\,\gperp\,\calL_2$. By \cite[Lemma I.5.3]{l}, $\calL_1[\ds]\calL_2$ is uniformly positive, and we have
\begin{equation}\label{e:decomp_easy}
\calH = \calL_1\,[\ds]\,\calL_2\,[\ds]\,\wt\calH,
\end{equation}
where $\wt\calH = (\calL_1[\ds]\calL_2)^\gperp$. Define $\wt A := A|\wt\calH$. Then, since $\sigma(A|\calL_j^\gperp)\subset\ol{\sigma(A)\setminus\Delta_j}$, $j=1,2$, we have due to \eqref{e:decomp_easy}
\begin{equation*}\label{e:hihi}
\sigma(\wt A)\subset\ol{\sigma(A)\setminus\Delta_1}\,\cap\,\ol{\sigma(A)\setminus\Delta_2} \;=\; \ol{\sigma(A)\setminus(\Delta_1\cup\Delta_2)}.
\end{equation*}
This implies $(a',a'')\cup (b'',b')\subset\rho(\wt A)$.
Let $\Gamma$ be a closed curve in $\rho(\wt A)$ which is symmetric with respect to the real axis such that the part of the spectrum of $\wt A$ in the interior of $\Gamma$ coincides with $\sigma(\wt A)\cap [a'',b'']$. The Riesz-Dunford projection
$$
\wt E := -\frac 1 {2\pi i}\,\int_\Gamma\,(\wt A - \lambda)^{-1}\,d\lambda
$$
is then easily seen to be a bounded $\product$-symmetric operator in $\wt\calH$. With respect to the decomposition \eqref{e:decomp_easy} of $\calH$ we now define
\begin{equation}\label{e:def_E}
E(\Delta) := I_{\calL_1} \,[\ds]\, I_{\calL_2} \,[\ds]\, \wt E.
\end{equation}
This is obviously a $G$-selfadjoint projection in $\calH$ which commutes with $A$.

The definition of $E(\Delta)$ in \eqref{e:def_E} depends on the choice of $a''$ and $b''$. However, it is not difficult to show that a different value for $a''$ leads to the same operator in \eqref{e:def_E}. The same then holds for a different $b''$ which proves that the definition \eqref{e:def_E} is in fact independent on the choice of $a''$ and $b''$.

For arbitrary $\Delta\in\mathscr{M}_S(a,b)$ we define $E(\Delta) := E(\delta_1) + \ldots + E(\delta_n)$, where the $\delta_j$ are the connected components of $\Delta$.

\
\\
\indent{\bf 2.} Let us prove that the set function $E$, defined in the first part of this proof, is in fact a local spectral function for $A$ on $(a,b)$. Let $\Delta\in\mathscr M_S(a,b)$ be an interval. Then it is evident that (S3) holds for $\Delta$, and we have (using the notation from part 1)
\begin{align*}
\sigma(A|E(\Delta)\calH)
&= \sigma(A|\calL_1)\cup\sigma(A|\calL_2)\cup\sigma(\wt A|\wt E\calH)\\
&\subset \ol{\sigma(A)\cap(\Delta_1\cup\Delta_2)}\;\cup\,\left(\sigma(A)\cap [a'',b'']\right)\\
&= \ol{\sigma(A)\cap(\Delta_1\cup\Delta_2\cup [a'',b''])},
\end{align*}
which is (S4). The property (S5) for $\Delta$ is proved similarly. Now, let us show that for two intervals $\Delta_1,\Delta_2\in\mathscr M_S(a,b)$ we have
\begin{equation}\label{e:spezial}
\Delta_1\cap\Delta_2 = \emptyset\quad\Lra\quad E(\Delta_1)E(\Delta_2) = 0.
\end{equation}
Indeed, if $\ol{\Delta_1}\cap\ol{\Delta_2} = \emptyset$, then, as $E(\Delta_1)E(\Delta_2)$ maps onto $E(\Delta_1)\calH\cap E(\Delta_2)\calH$, we have $\sigma(A|E(\Delta_1)E(\Delta_2)\calH)\subset\ol{\Delta_1}\cap\ol{\Delta_2} = \emptyset$, and thus $E(\Delta_1)E(\Delta_2) = 0$. It remains to consider the case that $\Delta_1$ and $\Delta_2$ have a common endpoint $\alpha$. But then, a real neighborhood of $\alpha$ must be of positive type, and the assertion follows from Theorem \ref{t:B}.

Due to \eqref{e:spezial} it suffices to prove (S1)--(S5) only for intervals $\Delta,\Delta_j\in\mathscr M_S(a,b)$, and hence it remains to prove (S1) and (S2) for intervals. But (S1) follows from \eqref{e:spezial} and (S2), so that only the proof of (S2) is left. For this, let $\Delta_j\in\mathscr{M}_S(a,b)$, $j\in\N$, be mutually disjoint  intervals such that $\Delta := \bigcup_{j=1}^\infty\,\Delta_j$ is also an element of $\mathscr{M}_S(a,b)$. Due to the definition of $E(\Delta)$ via connected components and the finiteness of $\Delta\cap S$, it is no restriction to assume that each $\Delta_j$ is an interval with $\Delta_j\cap S = \emptyset$. Hence, also $\Delta\cap S = \emptyset$. Therefore, the subspace $\hat\calH := E(\Delta)\calH$ is uniformly positive and the operator $\hat A := A|\hat\calH$ is a bounded selfadjoint operator in the Hilbert space $(\hat\calH,\product)$ with $\sigma(\hat A)\subset\ol\Delta$. Now, the assertion follows from the fact that the restriction of $E(\Delta_j)$ to $\hat\calH$ coincides with $\hat E(\Delta_j)$, where $\hat E$ is the usual spectral measure of $\hat A$ in $\hat\calH$.

\
\\
\indent{\bf 3.} In this step we will show that $E(\Delta)\calH$ is an Almost Pontryagin space with finite rank of non-positivity. It is sufficient to show this for a compact interval $\Delta\in\mathscr{M}_S(a,b)$ such that $S\cap\Delta$ consists only of one point $\gamma$. Let $\calL_{\gamma}(A) = \calL_+\,[\ds]\,\calL_-\,[\ds]\,\calL_0$ be a fundamental decomposition of the algebraic eigenspace $\calL_{\gamma}(A)$ which is by Theorem \ref{t:big_one} an Almost Pontryagin space with finite rank of non-positivity. By Lemma \ref{l:APS}, we find subspaces $\calP,\calM\subset E(\Delta)\calH$ with $\dim\,\calP < \infty$ such that
$$
E(\Delta)\calH= \calL_+\,[\ds]\,\calL_-\,[\ds]\,\calL_{00}\,[\ds]\,(\calL_{01}\,\ds\,\calP)\,[\ds]\,\calM,
$$
where
$$
\calL_{00} = \calL_{S\cap\Delta}(A)\cap\calH^\iso,\quad\calL_{\gamma}(A)^\iso = \calL_{00}\,\ds\,\calL_{01},
$$
and $\calL_{\gamma}(A)^\gperp = \calL_{\gamma}(A)^\iso\,[\ds]\,\calM$. Hence, with respect to the decomposition
$$
E(\Delta)\calH = \calL_{\gamma}(A) \,\ds\, \calM \,\ds\, \calP
$$
the operator $A_\Delta := A|E(\Delta)\calH$ admits the following representation:
$$
A_\Delta = \begin{pmatrix}A_{11} & A_{12} & A_{13}\\ 0 & A_{22} & A_{23}\\ 0 & 0 & A_{33}\end{pmatrix}.
$$
Obviously, the operator $A_{22}$ is  $\product$-symmetric in the subspace $\calM$.
Let us show that $\sigma(A_{22}) = \spp(A_{22})$. Then from \cite[Theorem 3.1]{lmm}
it follows that $\calM$ is uniformly positive and thus that $E(\Delta)\calH$ is an
Almost Pontryagin space with finite rank of non-positivity. We have
$\sigma(A_\Delta) = \spip(A_\Delta)$, and since the operator $A_{12}$ maps into the
 finite-dimensional subspace $\calL_{\gamma}(A)^\iso$, it follows from Theorem
\ref{t:closed subspace} that $\sigma(A_{22}) = \spip(A_{22})$. For
$\sigma(A_{22}) = \spp(A_{22})$ it suffices to show that $\ker(A_{22} - \lambda)$
is positive for all $\lambda\in\sigma(A_{22})$; cf.\ Corollary \ref{c:spp_in_spip}.
 Hence, let $\lambda\in\sigma(A_{22})$, and let $x\in\dom A_{22}\subset\calM$ with
$(A_{22} - \lambda)x = 0$. If $\lambda = \gamma$ we have
$(A_\Delta - \gamma)x = A_{12}x\in\calL_{\gamma}(A)$ and hence $(A - \gamma)^{k}x = 0$
 for some $k\in\N$. But this implies $x\in\calL_{\gamma}(A)$ and therefore $x = 0$.
  Let $\lambda\ne \gamma$. There is $m\in\N$ such that
  $(A - \gamma)^{m}\calL_{\gamma}(A) = \{0\}$. Set
$$
y := (A - \gamma)^{m}x.
$$
Then, since $(A_\Delta - \lambda)x = A_{12}x\in\calL_{\gamma}(A)$, we have $(A_\Delta - \lambda)y = 0$ and hence either $y = 0$ (which implies $x = 0$) or $[y,y] > 0$ as $\lambda\in\spp(A_\Delta)$. Suppose $y\neq 0$. Then we have $(A - \gamma)x = A_{12}x + (\lambda - \gamma)x$, and, by induction,
$$
(A - \gamma)^{2m}\,x = \ell + (\lambda - \gamma)^{2m}\,x
$$
with some $\ell\in\calL_{\gamma}(A)$. Finally, we obtain
\begin{align*}
[x,x]
&= \frac{\left[ (A - \gamma)^{2m}\,x - \ell\,,x \right]}{(\lambda - \gamma)^{2m}}
 = \frac{\left[(A - \gamma)^{2m}\,x\,,x \right]}{(\lambda - \gamma)^{2m}} =
  \frac{[y,y]}{(\lambda - \gamma)^{2m}} > 0,
\end{align*}
and the theorem is proved.
\end{proof}

\vspace*{1cm}
\section*{Contact information}

\vspace*{.3cm}
\noindent
{\bf Friedrich Philipp}

Institut f\"ur  Mathematik,  Technische Universit\"{a}t Berlin

Stra\ss e des 17.\ Juni 136, D-10623 Berlin, Germany

http://www.tu-berlin.de/?fmphilipp

philipp@math.tu-berlin.de

\ \\
\noindent
{\bf Carsten Trunk }

Institut f\"ur  Mathematik,  Technische Universit\"{a}t Ilmenau

Postfach 100565, D-98684 Ilmenau, Germany

http://www.tu-ilmenau.de/analysis/team/carsten-trunk

carsten.trunk@tu-ilmenau.de

\end{document}